\pdfoutput=1
\documentclass[a4paper,11pt]{article}

\usepackage{amsmath,amsthm,amssymb,amsfonts}
\usepackage{mathrsfs}
\usepackage{geometry}
\usepackage{hyperref}
\usepackage{cleveref}
\usepackage{graphicx}
\usepackage{mathtools}
\usepackage[british]{babel}
\usepackage{tikz}
\usepackage{xcolor}

\makeatletter
\newtheorem*{rep@theorem}{\rep@title}
\newcommand{\newreptheorem}[2]{%
\newenvironment{rep#1}[1]{%
 \def\rep@title{#2 \ref{##1}}%
 \begin{rep@theorem}}%
 {\end{rep@theorem}}}
\makeatother

\newtheorem{thm}{Theorem}[section]
\newtheorem{cor}[thm]{Corollary}
\newtheorem{lem}[thm]{Lemma}
\newtheorem{prop}[thm]{Proposition}

\newtheorem{ques}[thm]{Question}
\newreptheorem{prop}{Proposition}

\newtheorem{thmintro}{Theorem}

\newtheorem{corintro}[thmintro]{Corollary}

\theoremstyle{remark}

\newtheorem{note}[thm]{Notation}

\theoremstyle{definition}
\newtheorem{rem}[thm]{Remark}
\newtheorem*{rem*}{Remark}
\newtheorem{defn}[thm]{Definition}
\newtheorem*{defn*}{Definition}
\newtheorem{exm}[thm]{Example}

\usepackage{xcolor}
\newcommand{\G}{\Gamma}

\newcommand{\sq}{\subseteq}
\newcommand{\acts}{\curvearrowright}
\newcommand{\Z}{\mathbb{Z}}
\newcommand{\mc}{\mathcal}
\newcommand{\mf}{\mathfrak}

\newcommand{\x}{\times}
\newcommand{\g}{\gamma}

\newcommand{\N}{\mathbb{N}}
\newcommand{\Q}{\mathbb{Q}}
\newcommand{\R}{\mathbb{R}}
\newcommand{\T}{\mathcal{T}}
\newcommand{\ra}{\rightarrow}
\newcommand{\Ra}{\Rightarrow}
\newcommand{\bigast}{\mathop{\scalebox{1.5}{\raisebox{-0.2ex}{$\ast$}}}}%
\DeclareMathOperator{\supp}{esupp}
\DeclareMathOperator{\stsupp}{esupp_{\infty}}

\DeclareMathOperator{\lk}{lk}
\DeclareMathOperator{\St}{st}

\DeclareMathOperator{\Fix}{Fix}

\DeclareMathOperator{\Isom}{Isom}
\usetikzlibrary {arrows.meta}

\renewcommand{\leq}{\leqslant}
\renewcommand{\geq}{\geqslant}

\title{Short positive loxodromics in graph products}
\author{Elia Fioravanti\thanks{Supported by Emmy Noether grant 515507199 of the Deutsche Forschungsgemeinschaft (DFG).} \and Alice Kerr\thanks{Supported by EPSRC grant EP/V027360/1 ``Coarse geometry of groups and spaces''.}}
\date{\today}

\begin{document}

\maketitle

\begin{abstract}
   We give a method for effectively generating generalised loxodromics in subgroups of graph products, using positive words. This has several consequences for the growth of subsets of these groups. In particular, we show that graph products of groups with strong product set growth properties also share those properties. We additionally show that the set of growth rates of a class of subgroups of any graph product of equationally noetherian groups is well-ordered.
\end{abstract}

\section{Introduction}

A countable group $G$ is called \textit{acylindrically hyperbolic} if it admits a non-elementary acylindrical action on a hyperbolic space $X$. Many algebraic properties of $G$ can be obtained from such an action \cite{Osin2016}, often via finding elements of $G$ that are loxodromic in $X$. Although a ``typical'' element of $G$ will indeed be loxodromic in $X$ \cite{Maher2018,Sisto2018}, 
it is not possible, in general, to find such a loxodromic element within balls of uniformly bounded radius in all Cayley graphs of $G$ \cite{Minasyan2019}. This creates problems when one seeks to make effective various results about acylindrically hyperbolic groups. 

The situation is better if one places some restrictions on the group $G$ and, in many of the classical examples of acylindrically hyperbolic groups, it is indeed possible to generate \emph{generalised loxodromics} (i.e.\ loxodromics in some acylindrical $G$--action on a hyperbolic space) by uniformly short words in any generating set of $G$. For instance, this is true when $G$ is hyperbolic \cite{Koubi1998}, relatively hyperbolic \cite{Xie2007}, a mapping class group \cite{Mangahas2013}, or virtually torsion-free hierarchically hyperbolic \cite{Zalloum-rank-rigidity}. In the first three cases, it is also true for large classes of subgroups \cite{Arzhantseva2006,Cui2021,Mangahas2013}.

In this paper, we study the problem of generating generalised loxodromics by short words when the group $G$ is a \emph{graph product} $\mc{G}_{\G}$; see \Cref{sub:graph_products} for definitions. More generally, we are interested in arbitrary subgroups of graph products, and in only using \emph{positive} words to generate loxodromics, i.e.\ words avoiding inverses of the generators. This is particularly natural in view of the applications to product set growth described at the end of the introduction.

Every graph product $\mc{G}_{\G}$ is obtained by assembling a collection of potentially ``arbitrarily bad'' vertex groups $\mc{G}_v$. For this reason, it is natural to only be interested in generalised loxodromics that are not conjugate into any of these vertex groups. In several applications, it is also useful to consider an even more restricted class of generalised loxodromics, namely those that are loxodromic in a particular acylindrical $\mc{G}_{\G}$--action on a quasi-tree: the \emph{contact graph} $\mc{C}(\mc{G}_{\G})$. The latter is neither a largest nor a universal acylindrical action of $\mc{G}_{\G}$ in general (in the sense of \cite{Abbott2019/2}), but it is both when the vertex groups are infinite and $\G$ is connected \cite{Valiunas2021}.

In conclusion, we are interested in the following two kinds of elements of $\mc{G}_{\G}$. We say that an element $g\in\mc{G}_{\G}$ is:
\begin{itemize}
    \item \emph{Regular} if $g$ is a generalised loxodromic\footnote{When the graph product is finitely generated, this is equivalent to asking that the element be \emph{Morse} with respect to the word metric on $\mc{G}_{\G}$ induced by a finite generating set (see \Cref{lem:regular_definitions}).} for $\mc{G}_{\G}$ and no power of $g$ is conjugate into a vertex group;
    \item \emph{Strongly irreducible} if $g$ acts loxodromically on the contact graph $\mc{C}(\mc{G}_{\G})$. 
\end{itemize}
Strongly irreducible elements were introduced and studied by Genevois under the name of ``irreducible elements'' \cite{Genevois2018}, but we prefer to reserve the latter terminology for a weaker concept (see \Cref{defn:irreducible}). When $\mc{G}_{\G}$ is a graph product of infinite groups, regular and strongly irreducible elements coincide; in particular, this is the case in right-angled Artin groups. However, regular elements are a broader class in general, for instance in right-angled Coxeter groups.

The following is our main result. We write $\dim(\G)$ for the largest cardinality of a clique in $\G$. We make no finite generation or countability assumptions on the vertex groups of the graph product.

\begin{thmintro}\label{thmintro:short_loxodromics}
    Let $\G$ be a finite graph, and let $\mc{G}_{\G}$ be a graph product. There exists an integer $N=N(\G)$ such that the following statements hold for any subset $U\sq\mc{G}_{\G}$. (Alternatively, $N$ can be taken to depend on $\dim(\G)$ and $|U|$.)
    \begin{enumerate}
        \item If the subgroup $\langle U\rangle$ contains a regular element, then there exists an integer $1\leqslant n\leqslant N$ such that the product $U^n$ contains a regular element.
        \item If the subgroup $\langle U\rangle$ contains a strongly irreducible element, then there exists $1\leqslant n\leqslant N$ such that $U^n$ contains a strongly irreducible element.
    \end{enumerate}
\end{thmintro}

Note that, when $U$ is finite, regular and strongly irreducible elements are Morse in the subgroup $\langle U\rangle$ by \cite{Sisto-MathZ}, as they are generalised loxodromics also for $\langle U\rangle$.

\begin{rem*}
    In the case where we take $N$ to depend on $\dim(\G)$ and $|U|$, both are needed, even for symmetric subsets of right-angled Artin groups (see \Cref{sec:sharpness} for examples).
\end{rem*}

When $U$ is not symmetric (equivalently, when we restrict to positive words), \Cref{thmintro:short_loxodromics} is new even for right-angled Artin groups. When $U$ is symmetric, the above results can, with some extra work, be deduced from Corollaries 6.4 and 6.7 in \cite{Casals2021}. When $U$ is an entire generating set of a right-angled Artin group, the above results also follow from a result communicated to us by Carolyn Abbott and Thomas Ng (see \Cref{Appendix2}).

The main motivation for not requiring symmetry in \Cref{thmintro:short_loxodromics} comes from growth, as there has recently been interest in understanding the growth of non-symmetric subsets of various groups, which we discuss below. It also seems to be a natural question, as previous restrictions to symmetric sets, both in graph products and other types of groups, were driven more by convenience than necessity.

In all applications of \Cref{thmintro:short_loxodromics} in this paper, we will only use the bound $N$ depending on $\G$; this is most natural when studying growth. However, one area where the alternative bound depending on $\dim(\G)$ and $|U|$ could prove useful is the study of sequences of homomorphisms from a fixed free group $F_k$ into the family of right-angled Artin groups; for instance, such sequences naturally arise when studying equational noetherianity of families of right-angled Artin groups. \Cref{thmintro:short_loxodromics} implies that, for any sequence $\rho_n$ of homomorphisms from $F_k$ into bounded-dimension right-angled Artin groups, there exist a (positive) element $g\in F_k$ and a subsequence $n_i$ such that $\rho_{n_i}(g)$ is contracting for all $i$.

As an immediate consequence of the right-angled Artin group case of \Cref{thmintro:short_loxodromics}, we obtain the following result that applies to the (virtually) special groups of Haglund and Wise \cite{Haglund-Wise2008}.

\begin{corintro}\label{corintro:special}
    There exists an integer $N=N(h)$ with the following property. Let $G$ be the fundamental group of a compact special cube complex $C$ with $h$ hyperplanes. For any subset $U\sq G\setminus\{1\}$, exactly one of the following holds:
    \begin{itemize}
        \item there exists $1\leqslant n\leqslant N$ such that $U^n$ contains a Morse element of $G$;
        \item $\langle U\rangle$ is virtually contained in a direct product of infinite subgroups of $G$.
    \end{itemize}
    Alternatively, $N$ can be taken to depend only on $\dim(C)$ and the cardinality $|U|$.
\end{corintro}

A related result is \cite[Theorem~A]{Zalloum-rank-rigidity}, which applies to the case when $\langle U\rangle=G$ and provides a bound $N$ depending on $G$ itself. We refer the reader to \Cref{cor:ah_equivalence2} below for an analogue of \Cref{corintro:special} when the special cube complex $C$ is not assumed to be compact.

\Cref{thmintro:short_loxodromics} is proved by constructing an element of ``large enough support'' (\Cref{prop:full_support}), where the support of $U\sq\mc{G}_{\G}$ is the smallest set of vertices $\Lambda\sq\G^{(0)}$ such that $U$ can be conjugated into $\mc{G}_{\Lambda}$. We construct such an element by showing that it is possible to combine (most of) the supports of two elements by a bounded-length positive word (\Cref{lem:combining_supp}), and then iterating this process. These results can be compared to \cite[Theorem 6.16]{Minasyan2015}, where it was proved that if the support of some $H\leqslant \mc{G}_{\G}$ is irreducible, then there exists an element $g\in H$ with the same support as $H$. At the cost of adding the assumption that $H$ is finitely generated, our results give an effective and more general version of their theorem. For instance, we show that the irreducible assumption is not needed when $\mc{G}_{\G}$ is torsion-free (\Cref{cor:torsion_free_full_support}).

These results about combining supports rely on the actions of $\mc{G}_{\G}$ on associated Bass-Serre trees. In the same spirit of combining elements, we prove that if a group $G$ acts by isometries on a collection of trees $T_1,\ldots,T_k$, and $g,h\in G$ are such that at least one of $g$ or $h$ acts loxodromically on each $T_i$, then there is a bounded length positive word in $\{g,h\}$ which is simultaneously loxodromic on every $T_i$, with the bound on the length only dependent on $k$ (see \Cref{Appendix1}). An effective version of \cite[Theorem 6.16]{Minasyan2015} can also be deduced from this result. It was previously known that some positive word in $\{g,h\}$ would give a simultaneous loxodromic element, even in the more general case of the actions being on hyperbolic spaces \cite{Clay2018}, however there was no bound on the length of this word.

\subsection{Applications to growth}

One motivation for finding short loxodromics in a group is that it is often a step in proving various growth results. For instance, they can be used to prove effective versions of the Tits Alternative (see \Cref{sec:TitsAlternative}). Another example is the following corollary, which is obtained by combining \Cref{thmintro:short_loxodromics} with \cite{Fujiwara2021} and \cite{Valiunas2021}. Similar results have recently been proved for the set of exponential growth rates of hyperbolic groups \cite{Fujiwara2020}, equationally noetherian relatively hyperbolic groups \cite{Fujiwara2021}, right-angled Artin groups \cite{Kerr2021a}, and free-by-cyclic groups \cite{Kudlinska2024}.

\begin{note}
    In the statement below, $\omega(H,S)$ denotes the exponential growth rate of $H$ with respect to $S$, $\supp(S)$ denotes the minimal $\Lambda\sq\G$ such that $S$ is contained in a conjugate of $\mc{G}_{\Lambda}$, and girth($\Lambda$) is the length of the shortest cycle in $\Lambda$. 
\end{note}

\begin{corintro}
    \label{corintro:well_ordered}
     Let $\G$ be a finite graph, and let $\mc{G}_{\G}$ be a graph product of equationally noetherian groups. The following is a well-ordered subset of $\R_{\geq 1}$.
    \begin{align*}
        \{ \omega(H,S) \mid & \text{ $S\sq\mc{G}_{\G}$ finite, $H=\langle S\rangle$, $\supp(S)$ neither a single vertex nor a join, and}
     \\ & \ \emph{girth}(\supp(S))\geqslant 6 \}.
    \end{align*} 
\end{corintro}

This could also be proved using \cite[Corollary 3.13]{Cohen2023}, however in our version the girth condition is only used to ensure that $\mc{G}_{\supp(S)}$ is equationally noetherian \cite[Theorem E]{Valiunas2021}, and could possibly be improved. We also note that it may be possible to extend the above to growth rates of subsemigroups, as in \cite[Section 6]{Fujiwara2020}, given that \Cref{thmintro:short_loxodromics} does not assume that the subset in question is symmetric.

Another application to growth is in the area of product set growth, which has attracted interest in recent years, see for example \cite{Safin2011,Delzant2020,Coulon2022,Cui2021,Wan2023}. By combining \Cref{thmintro:short_loxodromics} with \cite[Corollary 3.2.20]{Kerr2021a}, we get the following result for general graph products.

\begin{corintro}\label{corintro:growth1}
    Let $\G$ be a finite graph, and let $\mc{G}_{\G}$ be a graph product. There exist constants $\alpha,\beta > 0$, only dependent on $\G$, such that for every finite $U \sq \mathcal{G}_{\G}$ for which $\emph{esupp}(U)$ is neither a single vertex nor a join, at least one of the following holds:
    \begin{itemize}
        \item $\langle U\rangle$ is isomorphic to $\Z$ or $D_{\infty}$;
        \item $|U^n| \geqslant (\alpha|U|)^{\beta n}$ for every $n \in \mathbb{N}$.
    \end{itemize}
\end{corintro}

Such an inequality naturally implies uniform exponential growth for the group $\langle U\rangle$. In fact, it is strictly stronger than uniform exponential growth, as it gives a lower bound on the exponential growth rate of $\langle U\rangle$ in terms of $|U|$. Product set growth also has applications to approximate groups, specifically to giving upper bounds on their size \cite{Button2013,Kerr2021a}.

Finally, we discuss an application to product set growth ``dichotomies''. More precisely, most classes of groups where product set growth is well-understood are also known to satisfy the following property, which is similar to the result of \Cref{corintro:growth1}.

\begin{defn*}
We say that a group $G$ satisfies a \emph{growth dichotomy} with parameters $\alpha,\beta>0$ if, for every finite $U\sq G$, at least one of the following must hold:
	\begin{enumerate}
        \item $\langle U\rangle$ has a finite index subgroup with infinite centre;
        \item $|U^n| \geqslant (\alpha|U|)^{\beta n}$ for every $n \in \mathbb{N}$.
	\end{enumerate}
\end{defn*}

Note that, for a finitely generated group, having a finite index subgroup with finite centre is an obstruction to all its generating sets having product set growth with uniform parameters $\alpha,\beta$; see \cite[Corollary 2.2.2]{Kerr2021a}. 

In \Cref{sec:growth}, we introduce a class $\mc{A}$ of ``admissible'' groups, each of which satisfies a growth dichotomy for some $\alpha,\beta$. The precise definition is a little stronger, see \Cref{defn:admissible} (this is required in the proof of \Cref{thmintro:growth_dichotomy} below). Importantly, the class $\mc{A}$ includes virtually abelian groups, hyperbolic groups, right-angled Artin groups (and more generally virtually special groups), free-by-cyclic groups, and Burnside groups of sufficiently large odd exponent (see \Cref{ex:good}). In addition, $\mc{A}$ is closed under taking virtual subgroups, and includes groups that are hyperbolic relative to groups in $\mc{A}$. If we relax the conditions on $\mc{A}$ to include groups that only satisfy a growth dichotomy for \emph{symmetric} subsets, then $\mc{A}$ also includes mapping class groups. 

By combining \Cref{thmintro:short_loxodromics} and \Cref{corintro:growth1}, we are able to prove that $\mc{A}$ is closed under taking graph products, which greatly generalises \cite[Theorem 1.0.4]{Kerr2021a}.

\begin{thmintro}
    \label{thmintro:growth_dichotomy}
    Let $\G$ be a finite graph, and let $\mc{G}_{\G}$ be a graph product of groups in $\mc{A}$. Let $G$ be a group that virtually embeds into $\mathcal{G}_{\G}$. Then $G$ is also in $\mc{A}$.
\end{thmintro}

\textbf{Plan of the paper:} In \Cref{sec:prelims}, we review basic properties of graph products and actions on trees. \Cref{sect:graph_product_main} contains the proofs of \Cref{thmintro:short_loxodromics} and \Cref{corintro:special}, while \Cref{sec:sharpness} provides examples showing that it is not possible to remove the dependence of our bounds on various parameters. Finally, \Cref{sec:growth} discusses the applications to growth, with \Cref{corintro:well_ordered} proved in \Cref{sec:well_ordered}, and \Cref{corintro:growth1} and \Cref{thmintro:growth_dichotomy} proved in \Cref{sub:product_set_growth}.

\smallskip
\textbf{Acknowledgements:} We would like to thank Anthony Genevois for suggesting the current quicker proof of \Cref{lem:regular_definitions}, Montserrat Casals-Ruiz for explaining some of her work with Ilya Kazachkov, Albert Garreta, and Javier de la Nuez Gonz\'{a}lez in \cite{Casals2019,Casals2021}, Monika Kudlinska for helping to clarify the free-by-cyclic cases in \Cref{ex:good}, Nicolas Vaskou for suggesting the inclusion of Dyer groups in \Cref{ex:good}, and both Carolyn Abbott and Thomas Ng for the statement and proof in \Cref{Appendix2}. We also thank the referee for their helpful comments.

\section{Preliminaries}\label{sec:prelims}

This section collects basic results and definitions on graph products and actions on trees. In \Cref{sub:graph_products}, we recall various properties of graph products shown in \cite{Antolin2015,Genevois2018,Valiunas2021}; we then define regular elements and characterise them in a few equivalent ways. The main goal of \Cref{sub:actions_on_trees} is to characterise exactly when a product of two isometries of a tree is elliptic; our main result (\Cref{prop:elliptic_product}) is certainly known to experts, but it does not seem to appear in the literature.

\subsection{Graph products}\label{sub:graph_products}

Consider a finite simplicial graph $\G$, together with a group $\mc{G}_v$ for each vertex $v\in\G$. We will work under the standing assumption that $\mc{G}_v\neq\{1\}$ for each $v\in\G$, but we will not normally assume that these subgroups are finitely generated or even countable.
 
The \emph{graph product} $\mc{G}_{\G}$ is the quotient of the free product $\bigast_{v\in\G}\mc{G}_v$ by the subgroup normally generated by all commutators of the form $[g,h]$, where $g\in\mc{G}_u$ and $h\in\mc{G}_w$ for adjacent vertices $u,w\in\G$. We refer to the groups $\mc{G}_v$ as the \emph{vertex groups} of the graph product. When all vertex groups are infinite cyclic, $\mc{G}_{\G}$ coincides with the right-angled Artin group $A_{\G}$; when all vertex groups have cardinality $2$, we instead recover the right-angled Coxeter group $W_{\G}$. Free and direct products of groups are also special cases of graph products, corresponding to the situations where the graph $\G$ is discrete or complete, respectively.

The following is a straightforward observation.

\begin{lem}\label{lem:finite_graph_product}
A graph product $\mc{G}_{\G}$ is finite if and only if $\G$ is a clique and all vertex groups are finite.
\end{lem}

As customary, we will not distinguish between a subset $\Lambda\sq\G^{(0)}$ and the full subgraph $\Lambda\sq\G$ with this vertex set. For every $\Lambda\sq\G$, we refer to the subgroup $\mc{G}_{\Lambda}:=\langle\mc{G}_v\mid v\in\Lambda\rangle$ as a \emph{standard parabolic} subgroup of $\mc{G}_{\G}$. The notation is not ambiguous, as the standard parabolic subgroup $\mc{G}_{\Lambda}\leqslant\mc{G}_{\G}$ is indeed isomorphic to the graph product $\mc{G}_{\Lambda}$. In fact, there is even a natural retraction $\mc{G}_{\G}\twoheadrightarrow\mc{G}_{\Lambda}$ vanishing on $\mc{G}_{\G\setminus\Lambda}$, which also shows that for any two subgraphs $\Lambda,\Lambda'\sq G$, we have $\mc{G}_{\Lambda}\leqslant\mc{G}_{\Lambda'}$ if and only if $\Lambda\sq\Lambda'$.

More generally, a subgroup of $\mc{G}_{\G}$ is called \emph{parabolic} if it is conjugate to a standard parabolic subgroup. The following fundamental property of parabolic subgroups was obtained by Antol\'in and Minasyan.

\begin{lem}\label{lem:parabolic_intersections}
\emph{\cite[Corollary~3.6 and Corollary~3.18]{Antolin2015}}
Intersections of parabolic subgroups are parabolic.
\end{lem}

As a consequence, the following notion is well-defined.

\begin{defn}
Let $S\sq\mc{G}_{\G}$ be a subset. The \emph{(essential) support} $\supp(S)$ is the minimal subset $\Lambda\sq\G$ such that $S$ is contained in a conjugate of the standard parabolic subgroup $\mc{G}_{\Lambda}$.
\end{defn}

Elements of $\mc{G}_{\G}$ admit a normal form. Specifically, every $g\in\mc{G}_{\G}$ can be written as $g=x_1\cdots x_k$, where each ``syllable'' $x_i$ is an element of a vertex group. If $k$ is smallest among all such writings of $g$, we speak of $x_1\cdots x_k$ being a \emph{reduced word} representing $g$. Any two reduced words representing $g$ differ by a sequence of swaps of consecutive syllables $x_i,x_{i+1}$ such that $x_i\in\mc{G}_{v_i},x_{i+1}\in\mc{G}_{v_{i+1}}$ and the vertices $v_i,v_{i+1}\in\G$ are joined by an edge of $\G$. This was first proved in E.\ Green's thesis \cite[Theorem~3.9]{Green}; also see \cite[Theorem~2.2]{Antolin2015}.

An element $g\in\mc{G}_{\G}$ is \emph{cyclically reduced} if it has the shortest normal form in its conjugacy class\footnote{It should be noted that Green calls such elements ``proper cyclically reduced'' in \cite[Definition~3.15]{Green}, reserving the terminology ``cyclically reduced'' for a weaker property.}. If $g\in\mc{G}_{\G}$ is an element and $g'$ is a cyclically reduced conjugate, then $\supp(g)$ is precisely the set of vertices $v\in\G$ such that the normal form of $g'$ involves a syllable in $\mc{G}_v$.

For a vertex $v\in\G$, we denote by $\lk(v)$ its \emph{link} in the graph $\G$ (i.e.\ the set of vertices adjacent to $v$), and by $\St(v)$ its \emph{star}, namely $\St(v):=\lk(v)\cup\{v\}$. If $\Delta\sq\G$ is a subset, it is convenient to also define
\[\Delta^{\perp}:=\bigcap_{v\in\Delta}\lk(v),\] 
the set of points adjacent to every vertex in $\Delta$. The following is another important property of parabolic subgroups.

\begin{lem}\label{lem:parabolic_normaliser}
\emph{\cite[Lemma~3.7 and Proposition~3.13]{Antolin2015}}
Given subsets $\Delta,\Lambda\sq\G$ and an element $g\in\mc{G}_{\G}$, we have $g\mc{G}_{\Delta}g^{-1}\leqslant\mc{G}_{\Lambda}$ if and only if $\Delta\sq\Lambda$ and $g\in \mc{G}_{\Lambda}\cdot\mc{G}_{\Delta^{\perp}}$. In particular, the normaliser of $\mc{G}_{\Lambda}$ in $\mc{G}_{\G}$ is the standard parabolic subgroup $\mc{G}_{\Lambda\cup\Lambda^{\perp}}\cong\mc{G}_{\Lambda}\x\mc{G}_{\Lambda^{\perp}}$.
\end{lem}

Since the vertex groups of our graph products can have torsion, it will often be more convenient to work with the following modified notion of support of an element.

\begin{defn}\label{defn:stable_support}
For an element $g\in\mc{G}_{\G}$, the \emph{stable support} $\stsupp(g)$ is the intersection $\bigcap_{n>0}\supp(g^n)$.
\end{defn}

In general, the stable support can be easily computed from the essential support using the next lemma. We say that a vertex $v\in\supp(g)$ is a \emph{cone vertex} if $\supp(g)\sq\St(v)$. When $v$ is a cone vertex, the standard parabolic subgroup $\mc{G}_{\supp(g)}$ splits as a direct product $\mc{G}_v\x\mc{G}_{\supp(g)\setminus\{v\}}$. In view of Lemma~\ref{lem:parabolic_normaliser}, this splitting is preserved by the action of the normaliser of $\mc{G}_{\supp(g)}$ and hence every conjugate of $\mc{G}_{\supp(g)}$ inherits a uniquely defined splitting. In particular, if $v\in\supp(g)$ is a cone vertex and $P$ is the conjugate of $\mc{G}_{\supp(g)}$ containing $g$, it makes sense to speak of the projection of $g$ to the only conjugate of $\mc{G}_v$ contained in $\mc{G}_{\supp(g)}$. We denote this projection by $\pi_v(g)$.

\begin{lem}\label{lem:stable_support}
For any element $g\in\mc{G}_{\G}$, we have $\supp(g)=F\sqcup\stsupp(g)$, where $F$ is the set of cone vertices $v\in\supp(g)$ such that $\pi_v(g)$ has finite order in $\mc{G}_v$.
\end{lem}
\begin{proof}
It is clear that $F$ is disjoint from $\stsupp(g)$, so it suffices to show that every vertex $v\in\supp(g)\setminus F$ lies in $\stsupp(g)$. If $v$ is a cone vertex and $\pi_v(g)$ has infinite order, then it is clear that $v\in\stsupp(g)$. If instead $v$ is not a cone vertex, then there exists $w\in\supp(g)$ such that $v$ and $w$ are not joined by an edge of $\G$. In this case, suppose without loss of generality that $g$ is cyclically reduced, with normal form $g=x_1\cdots x_k$. Let $i$ and $j$ be indices such that $x_i\in\mc{G}_v$ and $x_j\in\mc{G}_w$; up to replacing $g$ with $g^{-1}$, we can assume that $i<j$. Then, in each word $(x_1\cdots x_k)^n$ with $n>0$, none of the syllables in $\mc{G}_v$ appearing among the rightmost $k(n-1)$ syllables can be pushed left past the first occurrence of $x_j$; hence the first occurrence of $x_i$ will never cancel out, and we have $v\in\supp(g^n)$ for all $n>0$ as required.
\end{proof}

As a consequence of the previous lemma, for every element $g\in\mc{G}_{\G}$ there exists an integer $n\geqslant 1$ such that $\supp(g^n)=\stsupp(g)$. Note, however, that there will not exist a uniform such integer $n$ if the vertex groups of the graph product have elements of arbitrarily large finite order.

A useful property of elements of graph products is that they can be uniquely split into pairwise-commuting ``irreducible components''. Recall that a graph $\G$ is said to split as a \emph{join} if we can write $\G^{(0)}=A\sqcup B$ for two non-empty subsets $A,B$ such that there are edges of $\G$ joining each element of $A$ to each element of $B$.

\begin{defn}\label{defn:irreducible}
We say that an element $g\in\mc{G}_{\G}$ is:
\begin{enumerate}
\item \emph{irreducible}\footnote{We emphasise that our notion of irreducibility is different from the one considered by Genevois in \cite[p14]{Genevois2018}; his ``irreducible elements'' are what we call ``strongly irreducible'' in the discussion below.} if the full subgraph $\supp(g)\sq\G$ does not split as a join;
\item \emph{stably irreducible} if the full subgraph $\stsupp(g)\sq\G$ does not split as a join.
\end{enumerate}
\end{defn}

By Lemma~\ref{lem:stable_support}, every irreducible element is stably irreducible, as its stable and essential supports coincide. However, stably irreducible elements need not be irreducible, except when the vertex groups are torsion-free.

Note that every element $g\in\mc{G}_{\G}\setminus\{1\}$ can be written as $g=g_1\cdots g_n$, where the $g_i$ are pairwise-commuting irreducible elements with pairwise-disjoint essential supports. More precisely, $\supp(g)\sq\G$ is the join of the full subgraphs spanned by the $\supp(g_i)$. The elements $g_i$ are unique up to permutations 
and we therefore refer to them as the \emph{irreducible components} of $g$.

The classes of elements appearing in Theorem~\ref{thmintro:short_loxodromics}, namely regular and strongly irreducible elements, each correspond to loxodromics in suitable acylindrical $\mc{G}_{\G}$--actions on hyperbolic spaces. We now give more practical definitions for these classes of elements and discuss the equivalence to the definitions given in the introduction.

\begin{defn}\label{defn:regular}
We say that an element $g\in\mc{G}_{\G}$ is:
\begin{enumerate}
\item \emph{strongly irreducible} if $g$ is neither conjugate into a vertex group, nor into a direct product of \emph{non-trivial} parabolic subgroups; 
\item \emph{regular} if no power of $g$ is conjugate into a vertex group or into a direct product of \emph{infinite} parabolic subgroups.
\end{enumerate}
\end{defn}

Equivalently, strongly irreducible elements are those for which $\supp(g)$ is nonempty, not a single vertex and not contained in a join. Similarly, regular elements are those for which $\stsupp(g)$ is nonempty, not a single vertex and not contained in a join of graphs generating infinite parabolic subgroups.

Using Lemma~\ref{lem:stable_support}, it is straightforward to see that strongly irreducible elements are regular. The converse holds if the vertex groups are infinite, as we now observe.

\begin{rem}
    If all the vertex groups $\mc{G}_v$ are infinite, then regular and strongly irreducible elements coincide. Indeed, if $g\in\mc{G}_{\G}$ is regular, we must have $\stsupp(g)=\supp(g)$. Otherwise $\stsupp(g)$ would be contained in a join by \Cref{lem:stable_support}, which would violate regularity as all (nontrivial) parabolic subgroups of $\mc{G}_{\G}$ are infinite by the assumption.
\end{rem}

Every graph product $\mc{G}_{\G}$ acts acylindrically on a natural quasi-tree, known as the \emph{contact graph} and denoted by $\mc{C}(\mc{G}_{\G})$ \cite[Corollary~C]{Valiunas2021}. The precise definition of the contact graph will not be important for our purposes. As mentioned in the introduction, the following is an equivalent characterisation of strongly irreducible elements. We recall that an isometry $g$ of a hyperbolic space $X$ is called \textit{loxodromic} if $\langle g\rangle$ has unbounded orbits, and fixes exactly two points on the Gromov boundary $\partial X$.

\begin{prop}
    \label{StronglyIrreducibleIsLoxodromic}
    An element $g\in\mc{G}_{\G}$ is strongly irreducible if and only if $g$ is loxodromic for the action on the contact graph $\mc{G}_{\G}\acts\mc{C}(\mc{G}_{\G})$.
\end{prop}
\begin{proof}
    The forwards direction can be deduced by combining 
    \cite[Proposition~3.9]{Genevois2022} with \cite[Proposition~4.24]{Genevois2018}, since $\mc{C}(\mc{G}_{\G})$ is defined as the contact graph of a quasi-median Cayley graph of $\mc{G}_{\G}$ (see e.g.\ \cite[Definition~1.1]{Valiunas2021}).
    
    The backwards direction follows from the definition of the contact graph $\mc{C}(\mc{G}_{\G})$. More specifically, if $g$ is not strongly irreducible, then it stabilises some subgraph of $\mc{C}(\mc{G}_{\G})$ with diameter at most $2$.
\end{proof}

Regarding regular elements, we have the following equivalent descriptions. The reader can consult \cite{ACGH} for a survey-like discussion of the Morse property.

\begin{lem}\label{lem:regular_definitions}
The following are equivalent characterisations of regular elements $g\in\mc{G}_{\G}$:
\begin{enumerate}
\item no power of $g$ is conjugate into a vertex group, nor lies in a direct product of infinite parabolic subgroups (this is Definition~\ref{defn:regular});
\item  $g$ is stably irreducible, $\stsupp(g)$ is nonempty and not a single vertex, and $\stsupp(g)^{\perp}$ is a (possibly empty) clique with finite vertex groups;
\item $g$ is loxodromic in an acylindrical $\mc{G}_{\G}$--action on a hyperbolic space and no power of $g$ is conjugate into a vertex group.
\end{enumerate}
If $\mc{G}_{\G}$ is finitely generated, the following is a fourth equivalent characterisation:
\begin{enumerate}
\setcounter{enumi}{3}
\item $g$ is Morse in $\mc{G}_{\G}$ (equipped with the word metric determined by any finite generating set) and no power of $g$ is conjugate into a vertex group.
\end{enumerate}
\end{lem}
\begin{proof}
The implication $(3)\Ra(4)$ follows from \cite[Theorem~1]{Sisto-MathZ} and \cite[Theorem~1.4]{Osin2016}. The implications $(3)\Ra(1)$ and $(4)\Ra (1)\Ra (2)$ are clear, using Lemma~\ref{lem:finite_graph_product} for the latter. 

We are left to show that $(2)\Ra(3)$, for which we will use a strategy due to Minasyan and Osin \cite{Minasyan2015}. Assuming condition~(2), we will show that $g$ is a WPD element in a $\mc{G}_{\G}$--action on a simplicial tree, after which another application of \cite[Theorem~1.4]{Osin2016} shows that $g$ is loxodromic in an acylindrical action on a hyperbolic space, as required.

Pick any vertex $v\in\stsupp(g)$ and let $\T_v$ be the Bass--Serre tree of the amalgamated-product splitting $\mc{G}_{\G}=\mc{G}_{\St(v)}\ast_{\mc{G}_{\lk(v)}}\mc{G}_{\G\setminus\{v\}}$. (We will study this tree in greater detail in \Cref{sect:graph_product_main}.) The element $g$ is not conjugate into $\mc{G}_{\G\setminus\{v\}}$ because $v\in\supp(g)$, and it is not conjugate into $\mc{G}_{\St(v)}$ because $g$ is stably irreducible and $\stsupp(g)$ contains $v$ without being just $\{v\}$. Thus, $g$ is loxodromic in $\T_v$; let $A\sq\T_v$ be its axis.

Let $P\leqslant\mc{G}_{\G}$ be the pointwise stabiliser of $A$. Since stabilisers of edges of $\T_v$ are parabolic and intersections of parabolics are parabolic (Lemma~\ref{lem:parabolic_intersections}), the subgroup $P$ is parabolic. The element $g$ normalises $P$, since $gA=A$, hence $g$ lies in a product of parabolics $P\x P^{\perp}$ by Lemma~\ref{lem:parabolic_normaliser} (if $P=x\mc{G}_{\Delta}x^{-1}$ for some $\Delta$ and $x$, we denote the subgroup $x\mc{G}_{\Delta^{\perp}}x^{-1}$ by $P^{\perp}$ for simplicity). Since $g$ is stably irreducible by assumption, a power $g^n$ lies in $P$ or $P^{\perp}$; since $g^n$ does not fix $A$ pointwise, we must have $g^n\in P^{\perp}$. Since $\stsupp(g)^{\perp}$ is a clique with finite vertex groups, again by assumption, it follows that $P$ is finite. Appealing to \cite[Corollary~4.3 and Lemma 6.12]{Minasyan2015}, this finally implies that $g$ is a WPD element in $\T_v$, as we wanted. This completes the proof of the lemma.
\end{proof}

In fact, upcoming work of Ciobanu and Genevois will show that regular elements in finitely generated graph products are even contracting with respect to words metrics \cite{Ciobanu-Genevois}. Thus, ``Morse'' can be replaced with ``contracting'' in part~(4) of \Cref{lem:regular_definitions}.

\begin{rem}
We have good reason to work with both regular and strongly irreducible elements in this article. As mentioned, strong irreducibles are loxodromic in the contact graph $\mc{C}(\mc{G}_{\G})$. In particular, they are loxodromic in an acylindrical action on a quasi-tree, where the quality of the acylindricity and of the quasi-tree are independent of the specific element $g\in\mc{G}_{\G}$. This will be essential for our applications to growth (\Cref{corintro:growth1}), which rely on the second-named author's arguments in \cite{Kerr2021a}. 

At the same time, regular elements are more general and they seem like a more natural class of elements in view of the description in Lemma~\ref{lem:regular_definitions}(4). Each of them is loxodromic in an acylindrical $\mc{G}_{\G}$--action on a quasi-tree, using \cite{Balasubramanya2017,Bestvina2019a}, but uniformly controlling the quality of these actions seems tricky, only using results readily available in the literature (though the construction in \cite[Definition~6.49]{Genevois2019} applied to the canonical quasi-median Cayley graph of $\mc{G}_{\G}$ should do the job).
\end{rem}

\subsection{Actions on trees}\label{sub:actions_on_trees}

Let $T$ be a tree\footnote{We will only use simplicial trees throughout the article, though there is no difference in this subsection if $T$ happens to be a real tree, with the exception of \Cref{lem:locally_elliptic}.}. At various points in the article, it will be important to understand precisely when a product of two isometries of $T$ is elliptic. In this subsection, we prove a result (Proposition~\ref{prop:elliptic_product}) fully characterising the situations in which this happens. All of this is fairly standard and can largely be deduced from \cite{Serre1980,Culler1987}.

If an isometry $g\in\Isom(T)$ fixes a point of $T$, we say that $g$ is \emph{elliptic} and denote by $\Fix(g)\sq T$ its set of fixed points. If $g$ is not elliptic, we say that $g$ is \emph{loxodromic}, and note that in a tree this is equivalent to our previous definition. In this case, $g$ leaves invariant a unique geodesic $A(g)\sq T$, called the \emph{axis}, and it translates along $A(g)$ by a distance $\tau(g)>0$, called the \emph{translation length}. When $g$ is elliptic, we set $\tau(g)=0$. 

If $\alpha\sq T$ is an arc (a bounded length geodesic), we denote by $\ell(\alpha)$ its length; we also write $[x,y]$ for the arc with endpoints $x,y\in T$.

\begin{defn}[Coherent arcs]\label{defn:coherent}
Consider two arcs $\alpha,\beta\sq T$ such that their union is contained in a larger arc. If $\overrightarrow{\alpha},\overrightarrow{\beta}$ are orientations on $\alpha,\beta$, we say that $\overrightarrow{\alpha}$ and $\overrightarrow{\beta}$ are \emph{coherent} if they are restrictions of the same orientation on an arc containing $\alpha\cup\beta$. Otherwise, we say that $\overrightarrow{\alpha}$ and $\overrightarrow{\beta}$ are \emph{incoherent}.

A \emph{coherent pair} is a pair $(g,\alpha)$ where $g\in\Isom(T)$ and $\alpha\sq T$ is an arc such that $g\alpha\cup\alpha$ is contained in a larger arc and, choosing an orientation $\overrightarrow{\alpha}$, the arcs $\overrightarrow{\alpha}$ and $g\overrightarrow{\alpha}$ are coherent.
\end{defn}

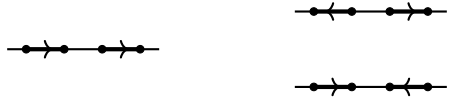
\begin{figure}[ht]
\centering
\begin{tikzpicture}[baseline=(current bounding box.center)]
\draw[thick] (-1,0) -- (1,0);
\draw[ultra thick] (-0.75,0) -- (-0.25,0);
\draw[ultra thick] (0.75,0) -- (0.25,0);
\draw[thick,->] (-0.5,0) -- (-0.4,0);
\draw[thick,->] (0.5,0) -- (0.6,0);
\draw[fill] (-0.75,0) circle [radius=0.05cm];
\draw[fill] (0.75,0) circle [radius=0.05cm];
\draw[fill] (0.25,0) circle [radius=0.05cm];
\draw[fill] (-0.25,0) circle [radius=0.05cm];
\node at (0,0.5) {};
\node at (0,-0.5) {};
\end{tikzpicture}
\hspace{1.5cm}
\begin{tikzpicture}[baseline=(current bounding box.center)]
\draw[thick] (-1,0.5) -- (1,0.5);
\draw[thick] (-1,-0.5) -- (1,-0.5);
\draw[ultra thick] (-0.75,0.5) -- (-0.25,0.5);
\draw[ultra thick] (0.75,0.5) -- (0.25,0.5);
\draw[thick,->] (-0.5,0.5) -- (-0.6,0.5);
\draw[thick,->] (0.5,0.5) -- (0.6,0.5);
\draw[fill] (-0.75,0.5) circle [radius=0.05cm];
\draw[fill] (0.75,0.5) circle [radius=0.05cm];
\draw[fill] (0.25,0.5) circle [radius=0.05cm];
\draw[fill] (-0.25,0.5) circle [radius=0.05cm];
\draw[ultra thick] (-0.75,-0.5) -- (-0.25,-0.5);
\draw[ultra thick] (0.75,-0.5) -- (0.25,-0.5);
\draw[thick,->] (-0.5,-0.5) -- (-0.4,-0.5);
\draw[thick,->] (0.5,-0.5) -- (0.4,-0.5);
\draw[fill] (-0.75,-0.5) circle [radius=0.05cm];
\draw[fill] (0.75,-0.5) circle [radius=0.05cm];
\draw[fill] (0.25,-0.5) circle [radius=0.05cm];
\draw[fill] (-0.25,-0.5) circle [radius=0.05cm];
\end{tikzpicture}
\caption{Coherent arcs (left) and incoherent arcs (right).}
\end{figure}

\begin{rem}\label{rmk:axis_coherent}
A pair $(g,\alpha)$ is coherent if and only if either $g$ is loxodromic and $\alpha\sq A(g)$, or $g$ is elliptic and $\alpha\sq\Fix(g)$. In particular, $g$ is loxodromic if and only if there exists a coherent pair $(g,\alpha)$ with $g\alpha\neq\alpha$.
\end{rem}

\begin{defn}[Creasing]\label{defn:creasing}
Consider $g\in\Isom(T)$ and an arc $\g\sq T$. We say that $g$ \emph{creases an arc $\alpha$ of $\g$} if we have $\alpha\cup g\alpha\sq\g$, the arcs $\alpha$ and $g\alpha$ share at most an endpoint and, choosing an orientation $\overrightarrow{\alpha}$, the arcs $\overrightarrow{\alpha}$ and $g\overrightarrow{\alpha}$ are incoherent.

When $\g$ is the axis of a loxodromic isometry $h$, we say that $g$ creases \emph{in the negative direction} if $g\alpha$ precedes $\alpha$ along $\g$, orienting $\g$ in the direction of translation of $h$.
\end{defn}

The following is the main result of this section.

\begin{prop}\label{prop:elliptic_product}
Consider $g,h\in\Isom(T)$. The product $gh$ is elliptic if and only if one of the following happens:
\begin{enumerate}
\item $g$ and $h$ are both elliptic in $T$ and $\Fix(g)\cap\Fix(h)\neq\emptyset$; 
\item\label{item:creasing} $\tau(h)<\tau(g)$ and $h$ creases an arc of $A(g)$ of length $\geqslant\frac{1}{2}(\tau(g)-\tau(h))$ in the negative direction (or the same situation with the roles of $g$ and $h$ swapped);
\item $g$ and $h$ are loxodromic with $\tau(g)=\tau(h)\leqslant\ell(A(g)\cap A(h))$ and $g,h$ translate in opposite directions along the intersection of their axes.
\end{enumerate}
\end{prop}

Before proving Proposition~\ref{prop:elliptic_product}, the following clarification can be helpful.

\begin{rem}\label{rmk:creasing_types}
Case~\ref{item:creasing} of Proposition~\ref{prop:elliptic_product} can only occur in one of two ways:
\begin{enumerate}
\item[(a)] $h$ is elliptic with $\Fix(h)\cap A(g)=\{x\}$ and $h$ maps an arc $\beta\sq A(g)$ starting at $x$ to an arc $h\beta$ that satisfies $\beta\cup h\beta\sq A(g)$ and $\beta\cap h\beta=\{x\}$;
\item[(b)] $g$ and $h$ are loxodromic with $\tau(h)=\ell(A(g)\cap A(h))<\tau(g)$, they translate in opposite directions along the intersection of their axes and, denoting by $x$ the point such that $A(g)\cap A(h)=[x,hx]$, the element $h$ maps an arc $\beta\sq A(g)\setminus A(h)$ based at $x$ to an arc $h\beta\sq A(g)\setminus A(h)$ based at $hx$. 
\end{enumerate}
See Figure~\ref{fig:elliptic_cases} for a schematic depiction of these configurations.
\end{rem}

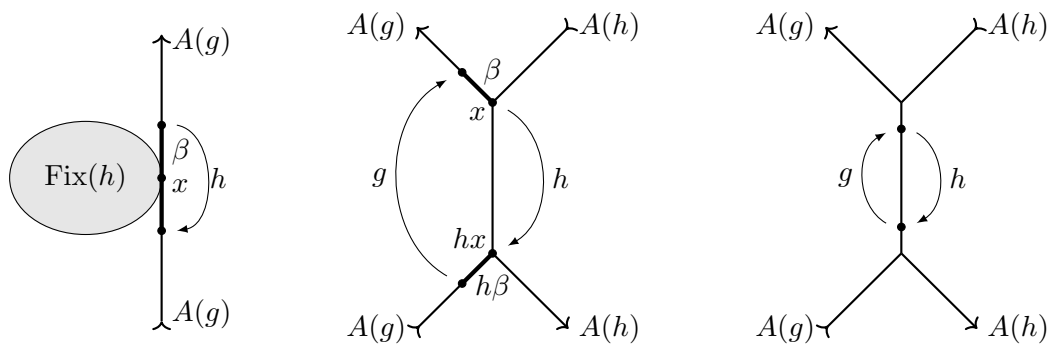
\begin{figure}[ht]
\centering
\begin{tikzpicture}[baseline=(current bounding box.center)]
\draw[thick,->] (0,-1.9) -- (0,1.9);
\draw[thick,->] (0,-1.94) -- (0,-1.9);
\draw[fill] (0,0) circle [radius=0.05cm];
\node[right] at (0,1.8) {$A(g)$};
\node[right] at (0,-1.8) {$A(g)$};
\filldraw[black,opacity=0.1,draw=black] (-1,0) ellipse (1cm and 0.75cm);
\draw (-1,0) ellipse (1cm and 0.75cm);
\node at (-1,0) {$\Fix(h)$};
\node[right] at (0,-0.1) {$x$};
\draw[ultra thick] (0,0) -- (0,0.7);
\draw[ultra thick] (0,0) -- (0,-0.7);
\draw[fill] (0,0.7) circle [radius=0.05cm];
\draw[fill] (0,-0.7) circle [radius=0.05cm];
\draw[-latex] (0.2,0.7) [out=-10, in=10] to (0.2,-0.7); 
\node[right] at (0.5,0) {$h$};
\node[right] at (0,0.35) {$\beta$};
\end{tikzpicture}
\hspace{1cm}
\begin{tikzpicture}[baseline=(current bounding box.center)]
\draw[thick] (0,-1) -- (0,1);
\draw[thick] (0,1) -- (1,2);
\draw[thick,->] (1,2) -- (0.96,1.96);
\draw[thick,->] (0,-1) -- (1,-2);
\draw[thick,->] (0,1) -- (-1,2);
\draw[thick] (0,-1) -- (-1,-2);
\draw[thick,->] (-1,-2) -- (-0.96,-1.96);
\node[left] at (-1,-2) {$A(g)$};
\node[left] at (-1,2) {$A(g)$};
\node[right] at (1,-2) {$A(h)$};
\node[right] at (1,2) {$A(h)$};
\draw[fill] (0,-1) circle [radius=0.05cm];
\draw[fill] (0,1) circle [radius=0.05cm];
\draw[-latex] (0.2,0.9) [out=-30, in=30] to (0.2,-0.9); 
\node[right] at (0.65,0) {$h$};
\draw[fill] (-0.4,-1.4) circle [radius=0.05cm];
\draw[fill] (-0.4,1.4) circle [radius=0.05cm];
\draw[-latex] (-0.6,-1.3) [out=150, in=210] to (-0.6,1.3); 
\node[left] at (-1.25,0) {$g$};
\draw[ultra thick] (0,1) -- (-0.4,1.4);
\draw[ultra thick] (0,-1) -- (-0.4,-1.4);
\node[above] at (0,1.1) {$\beta$};
\node[below] at (0,-1.15) {$h\beta$};
\node at (-0.3,-0.8) {$hx$};
\node at (-0.2,0.85) {$x$};
\end{tikzpicture}
\hspace{1cm}
\begin{tikzpicture}[baseline=(current bounding box.center)]
\draw[thick] (0,-1) -- (0,1);
\draw[thick] (0,1) -- (1,2);
\draw[thick,->] (1,2) -- (0.96,1.96);
\draw[thick,->] (0,-1) -- (1,-2);
\draw[thick,->] (0,1) -- (-1,2);
\draw[thick] (0,-1) -- (-1,-2);
\draw[thick,->] (-1,-2) -- (-0.96,-1.96);
\node[left] at (-1,-2) {$A(g)$};
\node[left] at (-1,2) {$A(g)$};
\node[right] at (1,-2) {$A(h)$};
\node[right] at (1,2) {$A(h)$};
\draw[fill] (0,-0.65) circle [radius=0.05cm];
\draw[fill] (0,0.65) circle [radius=0.05cm];
\draw[-latex] (-0.2,-0.6) [out=150, in=210] to (-0.2,0.6); 
\draw[-latex] (0.2,0.6) [out=-30, in=30] to (0.2,-0.6); 
\node[left] at (-0.5,0) {$g$};
\node[right] at (0.5,0) {$h$};
\end{tikzpicture}
\caption{From left to right, Cases~(2a), (2b) and (3) of Proposition~\ref{prop:elliptic_product} and Remark~\ref{rmk:creasing_types}.}
\label{fig:elliptic_cases}
\end{figure}

\begin{proof}[Proof of Proposition~\ref{prop:elliptic_product}]
In each of the cases listed in the proposition, it is straightforward to see that the product $gh$ is elliptic. (For this, it is helpful to remember that $gh$ is elliptic if and only if its inverse $h^{-1}g^{-1}$ is.) 

Conversely, suppose that $g$ and $h$ do not fall into any of the three cases of the proposition. We will find an arc $\alpha$ such that the pair $(gh,\alpha)$ is coherent and $gh\alpha\neq\alpha$. As explained in Remark~\ref{rmk:axis_coherent}, this implies that $gh$ is loxodromic.

The proof will require going through several cases. In the interest of clarity, we simply explain how to define the arc $\alpha$ in each of them, leaving to the reader the straightforward coherence check. Figures~\ref{fig:elliptic/loxodromic} and~\ref{fig:loxodromic/loxodromic} will be helpful for this process.

If $g,h$ are elliptic and $\Fix(g)\cap\Fix(h)=\emptyset$, define $\alpha$ as the shortest arc joining $\Fix(g)$ and $\Fix(h)$. See the leftmost configuration in Figure~\ref{fig:elliptic/loxodromic}.

Suppose that one of $g,h$ is elliptic and the other is loxodromic. Up to replacing $gh$ with its inverse $h^{-1}g^{-1}$, we can assume that $g$ is loxodromic and $h$ is elliptic. If $A(g)$ and $\Fix(h)$ are disjoint, define $\alpha$ as the shortest arc joining them; this case is analogous to the elliptic/elliptic one. If $A(g)\cap\Fix(h)$ is a non-trivial arc, define $\alpha$ as any of its sub-arcs. Finally, if $A(g)\cap\Fix(h)$ is a single point $p$, let $q\in A(g)$ be the point farthest from $p$ in the positive direction such that $hq\in A(g)$. If we are not in Case~2 of the proposition, then $d(p,q)<\tau(g)/2$ and we define $\alpha$ to be an arc contained in $A(g)$, based at $q$, moving in the positive direction of $A(g)$, and having $\ell(\alpha)<\tau(g)-2d(p,q)$. See Figure~\ref{fig:elliptic/loxodromic}.

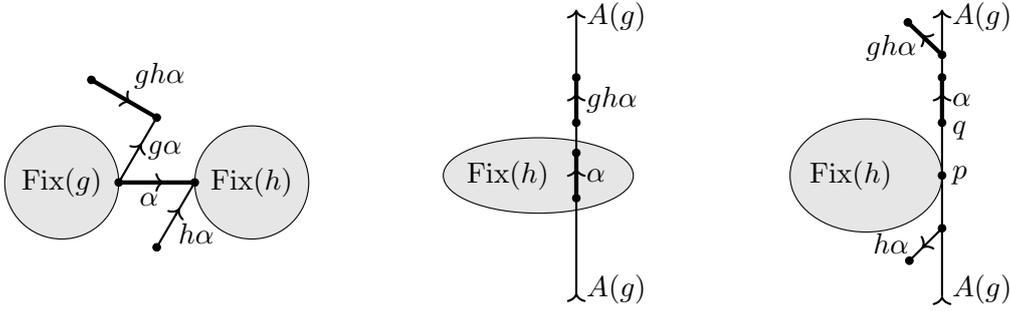
\begin{figure}[ht]
\centering
\begin{tikzpicture}[baseline=(current bounding box.center)]
\draw[ultra thick] (-0.5,0) -- (0.5,0);
\draw[thick,->] (0,0) -- (0.1,0);
\node[below] at (-0.1,0) {$\alpha$};
\draw[fill] (0.5,0) circle [radius=0.05cm];
\draw[fill] (-0.5,0) circle [radius=0.05cm];
\filldraw[black,opacity=0.1,draw=black] (-1.25,0) ellipse (0.75cm and 0.75cm);
\draw (-1.25,0) ellipse (0.75cm and 0.75cm);
\node at (-1.25,0) {$\Fix(g)$};
\filldraw[black,opacity=0.1,draw=black] (1.25,0) ellipse (0.75cm and 0.75cm);
\draw (1.25,0) ellipse (0.75cm and 0.75cm);
\node at (1.25,0) {$\Fix(h)$};
\draw[thick] (0.5,0) -- (0,-0.86);
\draw[thick,->] (0.25,-0.43) -- (0.3,-0.344);
\node[below right] at (0.15,-0.4) {$h\alpha$};
\draw[thick] (-0.5,0) -- (0,0.86);
\draw[thick,->] (-0.25,0.43) -- (-0.2,0.516);
\node[right] at (-0.25,0.43) {$g\alpha$};
\draw[ultra thick] (-0.86,1.36) -- (0,0.86);
\draw[thick,->] (-0.43,1.11) -- (-0.344,1.06);
\node[above right] at (-0.43,1.11) {$gh\alpha$};
\draw[fill] (0,0.86) circle [radius=0.05cm];
\draw[fill] (0,-0.86) circle [radius=0.05cm];
\draw[fill] (-0.86,1.36) circle [radius=0.05cm];
\end{tikzpicture}
\hspace{1.5cm}
\begin{tikzpicture}[baseline=(current bounding box.center)]
\draw[thick,->] (0,-1.9) -- (0,1.9);
\draw[thick,->] (0,-1.9) -- (0,-1.86);
\filldraw[black,opacity=0.1,draw=black] (-0.5,-0.3) ellipse (1.25cm and 0.5cm);
\draw (-0.5,-0.3) ellipse (1.25cm and 0.5cm);
\draw[ultra thick] (0,-0.6) -- (0,0);
\draw[thick,->] (0,-0.3) -- (0,-0.2);
\draw[fill] (0,0) circle [radius=0.05cm];
\draw[fill] (0,-0.6) circle [radius=0.05cm];
\draw[ultra thick] (0,0.4) -- (0,1);
\draw[fill] (0,1) circle [radius=0.05cm];
\draw[fill] (0,0.4) circle [radius=0.05cm];
\draw[thick,->] (0,0.7) -- (0,0.8);
\node[right] at (0,-1.8) {$A(g)$};
\node[right] at (0,1.8) {$A(g)$};
\node at (-0.9,-0.3) {$\Fix(h)$};
\node[right] at (0,-0.3) {$\alpha$};
\node[right] at (0,0.7) {$gh\alpha$};
\end{tikzpicture}
\hspace{1.5cm}
\begin{tikzpicture}[baseline=(current bounding box.center)]
\draw[thick,->] (0,-1.6) -- (0,2.2);
\draw[thick,->] (0,-1.64) -- (0,-1.6);
\draw[fill] (0,0) circle [radius=0.05cm];
\node[right] at (0,2.1) {$A(g)$};
\node[right] at (0,-1.5) {$A(g)$};
\filldraw[black,opacity=0.1,draw=black] (-1,0) ellipse (1cm and 0.75cm);
\draw (-1,0) ellipse (1cm and 0.75cm);
\node at (-1.2,0) {$\Fix(h)$};
\draw[fill] (0,0.7) circle [radius=0.05cm];
\draw[fill] (0,-0.7) circle [radius=0.05cm];
\node[right] at (0,0) {$p$};
\node[right] at (0,0.6) {$q$};
\draw[ultra thick] (0,0.7) -- (0,1.3);
\draw[thick,->] (0,1) -- (0,1.1); 
\draw[fill] (0,1.3) circle [radius=0.05cm];
\node[right] at (0,1) {$\alpha$};
\draw[thick] (0,-0.7) -- (-0.43,-1.13);
\draw[thick,->] (-0.215,-0.915) -- (-0.265,-0.965);
\draw[fill] (-0.43,-1.13) circle [radius=0.05cm];
\node[left] at (-0.3,-0.9) {$h\alpha$};
\draw[ultra thick] (0,1.6) -- (-0.45,2.03);
\draw[fill] (0,1.6) circle [radius=0.05cm];
\draw[fill] (-0.45,2.03) circle [radius=0.05cm];
\draw[thick,->] (-0.215,1.815) -- (-0.265,1.865);
\node[below left] at (-0.2,2) {$gh\alpha$};
\end{tikzpicture}
\caption{The elliptic/elliptic and elliptic/loxodromic cases in the proof of Proposition~\ref{prop:elliptic_product}.}
\label{fig:elliptic/loxodromic}
\end{figure}

Finally, suppose that both $g$ and $h$ are loxodromic. If $A(g)$ and $A(h)$ are disjoint, we define $\alpha$ as the arc connecting them. If instead $A(g)$ and $A(h)$ intersect each other, let $p$ be the last point of $A(g)\cap A(h)$, where the word ``last'' is meant with respect to the orientation on $A(g)$ given by the direction of translation of $g$.
Up to replacing $gh$ with $h^{-1}g^{-1}$, we can assume that $\tau(h)\leqslant\tau(g)$. Now, in each of the following cases: 
\begin{itemize}
\item if $A(g)\cap A(h)$ is a single point,
\item or if $g,h$ translate in the same direction along the intersection of their axes,
\item or if $\tau(h)>\ell(A(g)\cap A(h))$,
\item or if $\tau(h)<\ell(A(g)\cap A(h))$ and $\tau(h)<\tau(g)$,
\end{itemize}
we define $\alpha$ as a sufficiently short arc (the required length bound varies in each of the above four situations) contained in $A(g)$, based at $p$ and moving in the positive direction of $A(g)$ from there. See the first and second configuration in Figure~\ref{fig:loxodromic/loxodromic} for the third and fourth of the above cases, respectively.

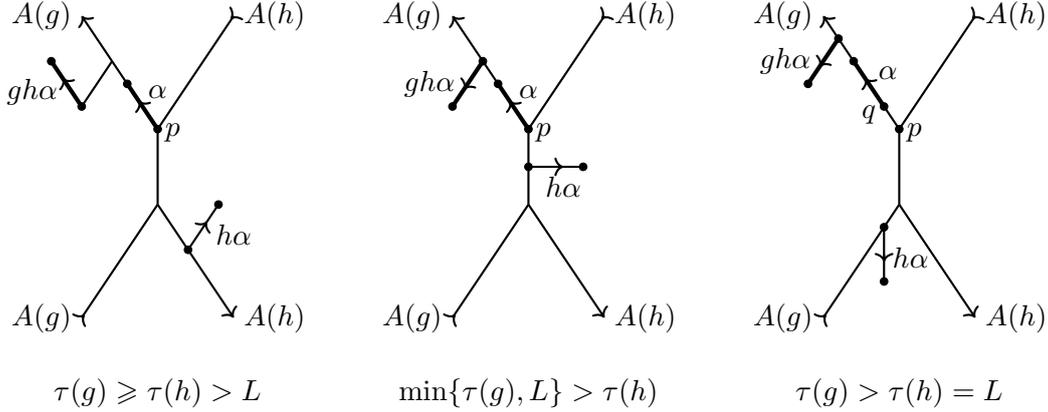
\begin{figure}[ht]
\centering
\begin{tikzpicture}[baseline=(current bounding box.center)]
\node at (0,-3) {$\tau(g)\geqslant\tau(h)>L$};
\draw[thick] (0,-0.5) -- (0,0.5);
\draw[thick] (0,0.5) -- (1,2);
\draw[thick,->] (1,2) -- (0.96,1.96);
\draw[thick,->] (0,-0.5) -- (1,-2);
\draw[thick,->] (0,0.5) -- (-1,2);
\draw[thick] (0,-0.5) -- (-1,-2);
\draw[thick,->] (-1,-2) -- (-0.96,-1.96);
\node[left] at (-1,-2) {$A(g)$};
\node[left] at (-1,2) {$A(g)$};
\node[right] at (1,-2) {$A(h)$};
\node[right] at (1,2) {$A(h)$};
\draw[ultra thick] (0,0.5) -- (-0.4,1.1);
\draw[fill] (0,0.5) circle [radius=0.05cm];
\draw[fill] (-0.4,1.1) circle [radius=0.05cm];
\draw[thick,->] (-0.2,0.8) -- (-0.25,0.875);
\draw[thick] (0.4,-1.1) -- (0.8,-0.5);
\draw[fill] (0.8,-0.5) circle [radius=0.05cm];
\draw[fill] (0.4,-1.1) circle [radius=0.05cm];
\draw[thick,->] (0.6,-0.8) -- (0.65,-0.725);
\draw[thick] (-0.6,1.4) -- (-1,0.8);
\draw[fill] (-1,0.8) circle [radius=0.05cm];
\node at (0,1) {$\alpha$};
\node at (1,-0.9) {$h\alpha$};
\draw[ultra thick] (-1,0.8) -- (-1.4,1.4);
\draw[fill] (-1.4,1.4) circle [radius=0.05cm];
\draw[thick,->] (-1.2,1.1) -- (-1.25,1.175);
\node at (-1.65,1) {$gh\alpha$};
\node at (0.2,0.45) {$p$};
\end{tikzpicture}
\hspace{.5cm}
\begin{tikzpicture}[baseline=(current bounding box.center)]
\node at (0,-3) {$\min\{\tau(g),L\}>\tau(h)$};
\draw[thick] (0,-0.5) -- (0,0.5);
\draw[thick] (0,0.5) -- (1,2);
\draw[thick,->] (1,2) -- (0.96,1.96);
\draw[thick,->] (0,-0.5) -- (1,-2);
\draw[thick,->] (0,0.5) -- (-1,2);
\draw[thick] (0,-0.5) -- (-1,-2);
\draw[thick,->] (-1,-2) -- (-0.96,-1.96);
\node[left] at (-1,-2) {$A(g)$};
\node[left] at (-1,2) {$A(g)$};
\node[right] at (1,-2) {$A(h)$};
\node[right] at (1,2) {$A(h)$};
\draw[ultra thick] (0,0.5) -- (-0.4,1.1);
\draw[fill] (0,0.5) circle [radius=0.05cm];
\draw[fill] (-0.4,1.1) circle [radius=0.05cm];
\draw[thick,->] (-0.2,0.8) -- (-0.25,0.875);
\node at (0,1) {$\alpha$};
\draw[ultra thick] (-0.6,1.4) -- (-1,0.8);
\draw[fill] (-1,0.8) circle [radius=0.05cm];
\draw[fill] (-0.6,1.4) circle [radius=0.05cm];
\draw[thick,->] (-0.8,1.1) -- (-0.85,1.025);
\node at (-1.3,1.1) {$gh\alpha$};
\draw[thick] (0,0) -- (0.72,0);
\draw[thick,->] (0.36,0) -- (0.46,0);
\draw[fill] (0,0) circle [radius=0.05cm];
\draw[fill] (0.72,0) circle [radius=0.05cm];
\node[below] at (0.46,0) {$h\alpha$};
\node at (0.2,0.45) {$p$};
\end{tikzpicture}
\hspace{.5cm}
\begin{tikzpicture}[baseline=(current bounding box.center)]
\node at (0,-3) {$\tau(g)>\tau(h)=L$};
\draw[thick] (0,-0.5) -- (0,0.5);
\draw[thick] (0,0.5) -- (1,2);
\draw[thick,->] (1,2) -- (0.96,1.96);
\draw[thick,->] (0,-0.5) -- (1,-2);
\draw[thick,->] (0,0.5) -- (-1,2);
\draw[thick] (0,-0.5) -- (-1,-2);
\draw[thick,->] (-1,-2) -- (-0.96,-1.96);
\node[left] at (-1,-2) {$A(g)$};
\node[left] at (-1,2) {$A(g)$};
\node[right] at (1,-2) {$A(h)$};
\node[right] at (1,2) {$A(h)$};
\draw[fill] (-0.2,0.8) circle [radius=0.05cm];
\node at (-0.4,0.7) {$q$};
\draw[ultra thick] (-0.2,0.8) -- (-0.6,1.4);
\draw[fill] (-0.6,1.4) circle [radius=0.05cm];
\draw[thick,->] (-0.4,1.1) -- (-0.45,1.175);
\node at (-0.15,1.25) {$\alpha$};
\draw[ultra thick] (-0.8,1.7) -- (-1.2,1.1);
\draw[fill] (-1.2,1.1) circle [radius=0.05cm];
\draw[fill] (-0.8,1.7) circle [radius=0.05cm];
\draw[thick,->] (-1,1.4) -- (-1.05,1.325);
\node at (-1.5,1.4) {$gh\alpha$};
\draw[thick] (-0.2,-0.8) -- (-0.2,-1.52); 
\draw[fill] (-0.2,-0.8) circle [radius=0.05cm];
\draw[fill] (-0.2,-1.52) circle [radius=0.05cm];
\draw[thick,->] (-0.2,-1.16) -- (-0.2,-1.26);
\node at (0.15,-1.2) {$h\alpha$};
\draw[fill] (0,0.5) circle [radius=0.05cm];
\node at (0.2,0.45) {$p$};
\end{tikzpicture}
\caption{Some of the loxodromic/loxodromic cases in the proof of Proposition~\ref{prop:elliptic_product}, where $L$ denotes the length of the intersection $A(g)\cap A(h)$.}
\label{fig:loxodromic/loxodromic}
\end{figure}

Excluding the situations falling into Case~3, we are only left to consider the situation when $\tau(h)=\ell(A(g)\cap A(h))<\tau(g)$ and $g,h$ translate in opposite directions. As in the elliptic/loxodromic case, let $q\in A(g)$ be the point farthest from $A(g)\cap A(h)$ for which we have $hq\in A(g)$, and note that $d(p,q)<\frac{1}{2}(\tau(g)-\tau(h))$ since we are not in Case~2. Define $\alpha$ as an arc contained in $A(g)$, based at $q$, moving in the positive direction from there, and having $\ell(\alpha)<\tau(g)-\tau(h)-2d(p,q)$. See the rightmost configuration in Figure~\ref{fig:loxodromic/loxodromic}. This completes the proof of the proposition.
\end{proof}

We will also record here the following observation for later use. To dispel any ambiguity in the infinitely generated case, we say that a group is \emph{elliptic} on a tree if its action has a global fixed point.

\begin{lem}\label{lem:locally_elliptic}
    Let $G\acts T$ be a group action on a tree. Suppose that for each descending chain of rays $\dots\supseteq r_n\supseteq r_{n+1}\supseteq\dots$ in $T$, the $G$--stabilisers of the $r_n$ eventually stabilise. Then, either $G$ is elliptic on $T$, or it contains a loxodromic element.
\end{lem}
\begin{proof}
    Suppose for the sake of contradiction that $G$ is not elliptic, but all elements of $G$ are. Then all finitely generated subgroups of $G$ are elliptic in $T$ by Serre's lemma \cite[p.\ 64]{Serre1980}. Pick a basepoint $p_0\in T$ and an element $g_0\in G$ that does not fix $p_0$. Set $G_0:=\langle g_0\rangle$ and let $F_0\sq T$ be the set of fixed points of $G_0$, which is a nonempty subtree of $T$. Let $p_1$ be the point of $F_0$ that is closest to $p_0$. Now, we repeat this procedure: let $g_1\in G$ be an element that does not fix $p_1$, set $G_1:=\langle g_0,g_1\rangle$, let $F_1$ be the set of fixed points of $G_1$, and define $p_2$ as the point of $F_1$ that is closest to $p_1$ (or equivalently to $p_0$). 
    
    Iterating, we obtain an infinite ascending sequence of finitely generated subgroups $\dots\leq G_n\leq G_{n+1}\leq\dots$ and a descending chain of subtrees $\dots\supseteq F_n\supseteq F_{n+1}\supseteq\dots$, where $F_n$ is the set of fixed points of $G_n$. We also obtain a ray $r\sq T$ based at $p_0\in T\setminus F_0$ by taking the limit of the arcs $\cdots\subseteq[p_0,p_n]\subseteq[p_0,p_{n+1}]\subseteq\cdots$, noting that $d(p_0,p_{n+1})=d(p_0,F_n)>n$ for all $n\geq 0$. 
    
    Let $r_n\sq r$ be the sub-ray based at $p_{n+1}$, and observe that $r_n\sq F_n$ and $r_{n+1}\sq r_n$. Denoting by $R_n$ the $G$--stabiliser of $r_n$, our assumption on chains of ray-stabilisers implies that there exists an integer $N$ such that $R_n$ is constant for $n\geq N$. Since $G_n\leq R_n$ by construction, it follows that $\bigcup_nG_n\leq R_N$. In particular, the subgroup $\bigcup_nG_n$ is elliptic in $T$. At the same time, note that the fixed set of $\bigcup_nG_n$ is contained in the intersection $\bigcap_nF_n$, which is empty because $d(p_0,F_n)>n$ for all $n\geq 0$ by construction. This is the required contradiction, which concludes the proof of the lemma.
\end{proof}

\begin{rem}
    This proof of \Cref{lem:locally_elliptic} does not work in general if $T$ is a real tree, as the points $p_n$ might not diverge. Thus, they might not trace a ray in the tree, but simply a half-open arc $[p_0,x)$ for some point $x\in T$. Nevertheless, the lemma and its proof do work for any real tree under the stronger assumption that the $G$--stabilisers of descending chains of arcs eventually stabilise.
\end{rem}

\section{Finding short regular and strongly irreducible elements}\label{sect:graph_product_main}

This section is devoted to the proof of \Cref{thmintro:short_loxodromics} (see \Cref{thm:short_loxodromics} below). The main ingredient is \Cref{lem:combining_supp}, which allows us to partially combine the supports of two elements of $\mc{G}_{\G}$ by a short positive word; more generally, given a subset $U\sq\mc{G}_{\G}$, we will construct elements of a small power $U^n$ containing most of $\supp(U)$ in their essential support (\Cref{prop:full_support}). 

These results are proved by studying the action of $\mc{G}_{\G}$ on the Bass--Serre trees of some natural amalgamated product splittings of $\mc{G}_{\G}$, and it will be important that these trees satisfy a property that we name ``bounded creasing'' (\Cref{prop:bounded_creasing}). The core of the argument is in \Cref{sub:combining}, which uses the bounded creasing property as a black box. \Cref{subsec:bounded_creasing} then quickly proves the bounded creasing property using results from \cite{Casals2019,Casals2021}. Finally, \Cref{sub:special} restricts to right-angled Artin groups and special groups, proving \Cref{corintro:special}.

\subsection{Combining element supports}\label{sub:combining}

Let $\mc{G}_{\G}$ be a graph product with $\G$ finite. Recall that $\dim(\G)$ is the largest cardinality of a clique in $\G$.

For each vertex $v\in\G$, there is a splitting of $\mc{G}_{\G}$ as the amalgamated product 
\[ \mc{G}_{\G} = \mc{G}_{\St(v)}\ast_{\mc{G}_{\lk(v)}}\mc{G}_{\G\setminus \{v\}} .\] 
Let $\T_v$ be the Bass--Serre tree of this splitting. We say that a vertex of $\T_v$ is of \emph{star type} if its stabiliser is conjugate to $\mc{G}_{\St(v)}$ and of \emph{complement type} otherwise, that is, if its stabiliser is conjugate to $\mc{G}_{\G\setminus \{v\}}$. For $g\in\mc{G}_{\G}$ and $v\in\G$, we denote by $\tau_v(g)$ the translation length of $g$ in $\T_v$. We write $A_v(g)$ for the axis of $g$ in $\T_v$ if $\tau_v(g)>0$, and $\Fix_v(g)$ for its fixed set if $\tau_v(g)=0$. 

\begin{rem}\label{rmk:HNN_tree}
When $\mc{G}_{\G}$ is a right-angled Artin group $A_{\G}$, it is usually simpler to work with slightly different Bass--Serre trees, namely those of the HNN splittings 
\[ A_{\G}=A_{\G\setminus\{v\}}\ast_{A_{\lk(v)}} .\] 
In the case of right-angled Artin groups, the entire discussion in \Cref{sect:graph_product_main} can be carried out with these trees in place of the $\T_v$, without requiring significant changes.
\end{rem}

Recall that a cone vertex of a subset $\Lambda\sq\G$ is a vertex $v\in\Lambda$ such that $\Lambda\sq\St(v)$. It is convenient to introduce the notation ${\rm acon}(\Lambda)$ for the \emph{aconical part} of $\Lambda$, namely the difference $\Lambda\setminus C$ where $C$ is the set of cone vertices of $\Lambda$. To be clear, we emphasise that $v$ is a cone vertex of $\{v\}$, and hence ${\rm acon}(\{v\})=\emptyset$. 

Given a subset $U\sq\mc{G}_{\G}$, we are interested in finding an element $g$ represented by a short word in $U$ such that $\stsupp(g)$ contains ${\rm acon}(\supp(U))$. This is a property that can be easily phrased in terms of loxodromics in the trees $\T_v$, as we now observe.

\begin{rem}\label{rmk:support_vs_tree}
    Straight from the definitions, an element $g\in\mc{G}_{\G}$ fixes a vertex of $\T_v$ of complement type if and only if $v\not\in\supp(g)$. Similarly, assuming that $v\in\supp(g)$, then $g$ fixes a vertex of $\T_v$ of star type if and only if $v$ is a cone vertex of $\supp(g)$. 
    
    As a consequence, $g$ is loxodromic in $\T_v$ if and only if $v\in{\rm acon}(\supp(g))$, and by \Cref{lem:stable_support} this occurs if and only if $v\in{\rm acon}(\stsupp(g))$.
\end{rem}

\begin{lem}\label{lem:U^2_covers_support}
   For any subset $U\sq\mc{G}_{\G}$, the following hold.
   \begin{enumerate}
       \item For each $v\in{\rm acon}(\supp(U))$, there exists $x\in U^2$ such that $v\in{\rm acon}(\stsupp(x))$.
       \item If $v$ is a cone vertex of $\supp(U)$, then there exists $u\in U$ such that $v\in\supp(u)$.
   \end{enumerate}
\end{lem}
\begin{proof}
We begin with part~(1). Since $v\in{\rm acon}(\supp(U))$, the subgroup $\langle U\rangle$ is not elliptic in $\T_v$: it does not fix a vertex of complement type because $v\in\supp(U)$, and it does not fix a vertex of star type because $v$ is not a cone vertex of $\supp(U)$. 

We now wish to appeal to \Cref{lem:locally_elliptic} to find a finite subset $U_0\sq U$ such that $\langle U_0\rangle$ is not elliptic in $\T_v$. For this, note that all edge-stabilisers for the action $\mc{G}_{\G}\acts\T_v$ are parabolic subgroups of $\mc{G}_{\G}$. The stabiliser of any ray in $\T_v$ is an intersection of such subgroups, hence it is itself parabolic by \Cref{lem:parabolic_intersections}. Chains of parabolic subgroups of $\mc{G}_{\G}$ have uniformly bounded length, as a consequence of \Cref{lem:parabolic_normaliser}, so \Cref{lem:locally_elliptic} can indeed be invoked to find $U_0$.

Since $U_0$ is finite and $\langle U_0\rangle$ is not elliptic in $\T_v$, Serre's lemma \cite[p.\ 64]{Serre1980} guarantees that either $U_0$ contains a loxodromic element, or it contains two elliptics with disjoint fixed sets. Either way an element $x\in U_0^2\sq U^2$ is loxodromic in $\T_v$, which guarantees that $v\in{\rm acon}(\stsupp(x))$.

Part~(2) is almost immediate. Since $v$ is a cone vertex, $U$ is contained in a conjugate of the product $\mc{G}_v\x\mc{G}_{\lk(v)}$. Since $v\in\supp(U)$, there must exist an element $u\in U$ with non-trivial projection to the $\mc{G}_v$--factor, which implies that $v\in\supp(u)$.
\end{proof}

We also record the following observation about axes of irreducible components.

\begin{rem}\label{rmk:tree_vs_components}
    Consider an element $g\in\mc{G}_{\G}$ and let $g=g_1\cdots g_k$ be its decomposition into irreducible components, as defined after \Cref{defn:irreducible}. The components $g_i$ pairwise commute and their essential supports form a partition of $\supp(g)$. In particular, at most one of the components $g_i$ can contain a given vertex $v\in\G$ in its essential support, and thus at most one of the $g_i$ can be loxodromic in the tree $\T_v$. If one component is loxodromic in $\T_v$, then its axis is fixed pointwise by the other components, so $g$ is loxodromic with the same axis and translation length. Finally, if all components are elliptic in $\T_v$, then their fixed sets intersect and $g$ is also elliptic in $\T_v$.

    In conclusion, for any $g\in\mc{G}_{\G}$ and $v\in\G$, the element $g$ is loxodromic in $\T_v$ if and only if one of its irreducible components is loxodromic in $\T_v$; in this case, a unique component $g_i$ of $g$ is loxodromic in $\T_v$, and we have $A_v(g_i)=A_v(g)$ and $\tau_v(g_i)=\tau_v(g)$.
\end{rem}

A key property of the Bass--Serre trees $\T_v$ is that only a limited amount of creasing (Definition~\ref{defn:creasing}) can occur in them, as the next result shows. This will allow us to invoke Proposition~\ref{prop:elliptic_product} to produce short (positive) words that are loxodromic in all trees $\T_v$ with $v\in{\rm acon}(\supp(U))$. The proof of this ``bounded creasing property'' is based on \cite{Casals2019,Casals2021} and we delay it until Subsection~\ref{subsec:bounded_creasing} below.

\begin{prop}[Bounded creasing]\label{prop:bounded_creasing}
Let $g\in\mc{G}_{\G}$ be loxodromic in $\T_v$ for some $v\in\G$.
\begin{enumerate}
\setlength\itemsep{.25em}
\item If $h\in\mc{G}_{\G}$ is an element without order--$2$ irreducible components, then $h$ cannot crease an arc of length $> (2\dim(\G)+2)\tau_v(g)$ of the axis $A_v(g)$.
\item Let $h\in\mc{G}_{\G}$ be another element that is loxodromic in $\T_v$. Let $\overline g,\overline h$ be the irreducible components of $g,h$ that are loxodromic in $\T_v$. If we have 
\[ \ell\big(A_v(g)\cap A_v(h)\big)> (2\dim(\G)+2)\max\{\tau_v(g),\tau_v(h)\}, \] 
then $\langle \overline g,\overline h\rangle\cong\Z$. 
\end{enumerate}
\end{prop}

We now use this result to prove a technical lemma and deduce \Cref{prop:full_support} below, which will construct short elements with large support. First, however, we need an observation about common refinements of two trees $\T_v,\T_w$.

\begin{rem}\label{rmk:double_splitting}
If two vertices $v,w\in\G$ do not span an edge, then $\mc{G}_{\G}$ splits as the graph of groups depicted below. Denote by $\T_{v,w}$ the Bass--Serre tree of this splitting, which has two orbits of edges and three orbits of vertices.

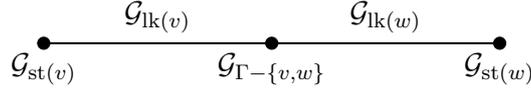
\begin{center}
\begin{tikzpicture}
\draw[fill] (0,0) circle [radius=0.08cm];
\draw[fill] (-3,0) circle [radius=0.08cm];
\draw[fill] (3,0) circle [radius=0.08cm];
\draw[thick] (-3,0) -- (3,0);
\node [below] at (-3,0) {$\mc{G}_{\St(v)}$};
\node [below] at (0,0) {$\mc{G}_{\G-\{v,w\}}$};
\node [below] at (3,0) {$\mc{G}_{\St(w)}$};
\node [above] at (-1.5,0) {$\mc{G}_{\lk(v)}$};
\node [above] at (1.5,0) {$\mc{G}_{\lk(w)}$};
\end{tikzpicture}
\end{center}

There are $\mc{G}_{\G}$--equivariant maps $\pi_v\colon\T_{v,w}\twoheadrightarrow\T_v$ and $\pi_w\colon\T_{v,w}\twoheadrightarrow\T_w$, each collapsing one of the two orbits of edges of $\T_{v,w}$. If an element $g\in\mc{G}_{\G}$ is loxodromic in $\T_v$, then it is loxodromic in $\T_{v,w}$. In this case, call $A_{v,w}(g)\sq\T_{v,w}$ its axis and $\tau_{v,w}(g)$ its translation length. If an edge $e\sq\T_{v,w}$ does not get collapsed in $\T_v$ (i.e.\ if $\pi_v(e)$ is an edge, rather than a single vertex), then we have $e\sq A_{v,w}(g)$ if and only if $\pi_v(e)\sq A_v(g)$; this is clear using the interpretation of axes in terms of coherent pairs laid out in Remark~\ref{rmk:axis_coherent}.

The following neat consequence will be needed below. Suppose that $g,h\in\mc{G}_{\G}$ are both loxodromic in both $\T_v$ and $\T_w$, where $v$ and $w$ do not span an edge. Suppose that $h^a$ creases a (non-degenerate) arc of $A_v(g)\sq\T_v$, while $h^b$ creases an arc of $A_w(g)\sq\T_w$ for some $a,b\in\N\setminus\{0\}$. Then the previous paragraph shows that both $h^a$ and $h^b$ crease arcs of $A_{v,w}(g)\sq\T_{v,w}$. This implies that $a=b$, since by Case (b) of \Cref{rmk:creasing_types} creasing can only occur when 
\[ a\cdot\tau_{v,w}(h)=\ell(A_{v,w}(g)\cap A_{v,w}(h))=b\cdot\tau_{v,w}(h) .\]
\end{rem}

The next result is a technical lemma (partially) combining the stable supports of two elements by a short word.

\begin{lem}\label{lem:combining_supp}
Consider two elements $g,h\in\mc{G}_{\G}$. There are three subsets $\Phi(g),\Phi(h)\sq\N$ and $\Psi(g,h)\sq\Q$, each of cardinality $\leqslant \dim(\G)$, such that
\[ {\rm acon}\big(\stsupp(g)\cup\stsupp(h)\big) \sq {\rm acon}(\stsupp(g^mh^n)) \]
for all integers $m,n\geqslant 4\dim(\G)+5$ such that $m\not\in\Phi(g)$, $n\not\in\Phi(h)$, $\frac{m}{n}\not\in\Psi(g,h)$ and such that $g^m$ and $h^n$ have no order--$2$ irreducible components.
\end{lem}

\begin{proof}
Setting $\Sigma:=\stsupp(g)\cup\stsupp(h)$, we have a partition
\[ \Sigma=\Sigma_{ee}\sqcup\Sigma_{e\ell}\sqcup\Sigma_{\ell\ell}, \]
where $\Sigma_{ee}$ contains the vertices $v\in\Sigma$ such that $g$ and $h$ are both elliptic in $\T_v$, the set $\Sigma_{e\ell}$ contains the vertices $v$ such that exactly one among $g$ and $h$ is elliptic in $\T_v$ and, finally, $\Sigma_{\ell\ell}$ is the set of vertices $v$ such that $g$ and $h$ are both loxodromic in $\T_v$. Let $C$ be the set of cone vertices of $\Sigma$.

To begin with, we show that $\Sigma_{ee}\setminus C\sq{\rm acon}(\stsupp(g^mh^n))$ for all $m,n>0$. Recalling Remark~\ref{rmk:support_vs_tree}, this amounts to showing, for each vertex $v\in\Sigma_{ee}\setminus C$, that the element $g^mh^n$ is loxodromic in $\T_v$. Consider a vertex $x\in\Fix_v(g^m)\cap\Fix_v(h^n)\sq\T_v$. Since $v$ lies in $\Sigma$, it lies in the stable support of at least one among $g$ and $h$, so $x$ cannot be a vertex of $\T_v$ of complement type. At the same time, $x$ cannot be of star type either, as its stabiliser would then be a conjugate of $\mc{G}_{\St(v)}$ containing $g^m$ and $h^n$, contradicting the assumption that $v\not\in C$. In conclusion $\Fix_v(g^m)\cap\Fix_v(h^n)=\emptyset$, hence $g^mh^n$ is loxodromic in $\T_v$ as required.

Next, we consider vertices $v\in\Sigma_{e\ell}$. Suppose for simplicity that $h$ is elliptic in $\T_v$ and $g$ is loxodromic, which we can ensure by replacing $gh$ with its inverse if necessary. If $m> 2(2\dim(\G)+2)$, then $\tau_v(g^m)>  2(2\dim(\G)+2)\tau_v(g)$. If in addition $h^n$ has no order--$2$ irreducible components, \Cref{prop:bounded_creasing}(1) guarantees that $h^n$ does not crease an arc of length $\geqslant\frac{1}{2}\tau_v(g^m)$ of $A_v(g)=A_v(g^m)$. Thus, the creasing cases of Proposition~\ref{prop:elliptic_product} do not occur, and the product $g^mh^n$ is loxodromic in $\T_v$. This shows that we have $\Sigma_{e\ell}\sq{\rm acon}(\stsupp(g^mh^n))$ for all $m,n\geqslant 4\dim(\G)+5$ such that $g^m$ and $h^n$ have no order--$2$ irreducible components.

Finally, we consider vertices in $\Sigma_{\ell\ell}$. First, we show that there are at most $\dim(\G)$ integers $m>0$ such that $g^m$ creases an arc of the axis of $h$ in some $\T_v$ with $v\in\Sigma_{\ell\ell}$. For simplicity, say that $(m,v)$ is a \emph{creasing pair} if $m>0$, $v\in\Sigma_{\ell\ell}$ and $g^m$ creases an arc of $A_v(h)$. Note that each vertex $v\in\Sigma_{\ell\ell}$ lies in at most one creasing pair $(m,v)$, as $m$ is necessarily the only integer for which $m\tau_v(g)=\ell(A_v(g)\cap A_v(h))$, by Case (b) of \Cref{rmk:creasing_types}. Thus, if more than $\dim(\G)$ integers were part of a creasing pair, there would exist creasing pairs $(m,v)$ and $(m',v')$ such that $m\neq m'$, $v\neq v'$ and $v,v'$ are not joined by an edge of $\G$. As observed at the end of Remark~\ref{rmk:double_splitting} considering the refinement $\mc{T}_{v,v'}$ of $\mc{T}_v$ and $\mc{T}_{v'}$, this is impossible.

In conclusion, there is a subset $\Phi(g)\sq\N$ with cardinality $|\Phi(g)|\leqslant \dim(\G)$, such that $g^m$ does not crease any axis $A_v(h)$ for $m\not\in\Phi(g)$ and $v\in\Sigma_{\ell\ell}$. In the same way, we obtain a subset $\Phi(h)\sq\N$ with the analogous property.

If $m\not\in\Phi(g)$ and $n\not\in\Phi(h)$, then the product $g^mh^n$ can be elliptic in some $\T_v$ with $v\in\Sigma_{\ell\ell}$ only if $g^m$ and $h^n$ fall into Case~3 of Proposition~\ref{prop:elliptic_product}, which requires that 
\[ \ell(A_v(g)\cap A_v(h))\geqslant\max\{\tau_v(g^m),\tau_v(h^n)\} .\] 
If we choose $m,n>2\dim(\G)+2$, then \Cref{prop:bounded_creasing}(2) guarantees that this can happen only if the irreducible components $\overline g,\overline h$ loxodromic in $\T_v$ satisfy $\langle\overline g,\overline h\rangle\cong\Z$. 

In this case, $g^mh^n$ is elliptic in $\T_v$ if and only if $\overline g^m\overline h^n=1$. Avoiding this simply amounts to ensuring that the ratio $\frac{m}{n}$ is not a specific rational number. Since $g$ and $h$ each have at most $\dim(\G)$ irreducible components, we only need to ensure that the ratio $\frac{m}{n}$ misses a subset $\Psi(g,h)\sq\Q$ with $|\Psi(g,h)|\leqslant \dim(\G)$.

As a recap, we have:
\begin{enumerate}
\item $\Sigma_{ee}\setminus C\sq{\rm acon}(\stsupp(g^mh^n))$ for all $m,n>0$;
\item $\Sigma_{e\ell}\sq{\rm acon}(\stsupp(g^mh^n))$ for all $m,n\geqslant 4\dim(\G)+5$ such that $g^m$ and $h^n$ have no order--$2$ irreducible components;
\item $\Sigma_{\ell\ell}\sq{\rm acon}(\stsupp(g^mh^n))$ for all $m,n\geqslant 2\dim(\G)+3$ with $m\not\in\Phi(g)$, $n\not\in\Phi(h)$ and $\frac{m}{n}\not\in\Psi(g,h)$.
\end{enumerate}
This completes the proof of the lemma.
\end{proof}

\begin{rem}
The typical situation forcing us to restrict to aconical parts in the statement of Lemma~\ref{lem:combining_supp} is when $g$ and $h$ are infinite-order elements of a vertex group $\mc{G}_v$ (or more generally, when they are elements of $\mc{G}_{\St(v)}$ with infinite-order projection to $\mc{G}_v$). It is easy to show that, up to avoiding one ratio $\frac{m}{n}$, the product $g^mh^n$ is not the identity. However, without placing any restrictions on the group $\mc{G}_v$, it can happen that there exist infinite-order elements $g,h\in\mc{G}_v$ such that the shortest product $g^mh^n$ with $m,n>0$ that again has infinite order requires extremely big integers $m,n$ (it is not hard to construct small cancellation groups generated by two elements $g,h$ with this property). 
In other words, ensuring that $v\in\stsupp(g^mh^n)$ is much more delicate than simply ensuring that $v\in\supp(g^mh^n)$.
\end{rem}

\begin{rem}\label{rmk:ugly_bound}
    Given two elements $g,h\in\mc{G}_{\G}$ and setting $d:=\dim(\G)\geqslant 1$, we can always find two integers $m,n$ satisfying all the requirements of Lemma~\ref{lem:combining_supp} with: 
        \[ \begin{cases} m,n\leqslant 6d+5, & \text{if the vertex groups of $\mc{G}_{\G}$ have no $2$--torsion;} \\ m,n\leqslant 2^{d+3}d, & \text{in general.} \end{cases} \]
    We now explain these estimates, beginning with the general one.

    For an integer $n\in\N$, define $\mf{p}_2(n)$ as the largest integer $k\in\N$ such that $2^k$ divides $n$. Let $\Delta(g)\sq\N$ be the set of orders of finite-order irreducible components of $g$. Recalling that each element of $\mc{G}_{\G}$ has at most $d$ irreducible components, it follows that $|\Delta(g)|\leqslant d$ and hence $|\mf{p}_2(\Delta(g))|\leqslant d$. Any interval of the form $[m+1,m+2^d]$ with $m\in\N$ contains integers realising all possible residues modulo $2^d$. In particular, such an interval always contains an integer $m'$ such that $\mf{p}_2(2m')\not\in\mf{p}_2(\Delta(g))$; equivalently, $g^{m'}$ has no order--$2$ irreducible components. 
    
    Now, let $\Phi(g),\Phi(h)\sq\N$ and $\Psi(g,h)\sq\Q$ be the sets provided by \Cref{lem:combining_supp} and recall that they all have cardinality at most $d$. Considering the pairwise disjoint intervals 
    \[ \big[4d+5+2^di,\ 4d+4+2^d(i+1)\big] \] 
    for $0\leqslant i\leqslant d$, we see that at least one of them must be disjoint from $\Phi(g)$. This shows that there exists an integer $m_0\in[4d+5,4d+4+2^d(d+1)]$ such that $m_0\not\in\Phi(g)$ and such that $g^{m_0}$ has no order--$2$ irreducible components. Observing that $|\Phi(h)\cup\Psi(g,h)|\leqslant 2d$, the same argument yields an integer $n_0\in[4d+5,4d+4+2^d(2d+1)]$ such that $n_0\not\in\Phi(h)$, $m_0/n_0\not\in\Psi(g,h)$ and such that $h^{n_0}$ has no order--$2$ irreducible components. Finally, it suffices to observe that $4d+4+2^d(2d+1)\leqslant 2^{d+3}d$ for all $d\geqslant 1$.

    If the vertex groups have no $2$--torsion, then no element of $\mc{G}_{\G}$ has order--$2$ irreducible components (recall e.g.\ \Cref{lem:stable_support}). In this case, we can first find an integer $m_0\in [4d+5,5d+5]$ with $m_0\not\in\Phi(g)$, and then an integer $n_0\in [4d+5,6d+5]$ with $n_0\not\in\Phi(h)$ and $m_0/n_0\not\in\Psi(g,h)$. This completes the proof of the claimed estimates.
\end{rem}

Using \Cref{lem:combining_supp} and \Cref{rmk:ugly_bound}, we can now construct elements in any subgroup $\langle U \rangle\leq \mc{G}_{\G}$ so that they have large support and are represented by a short word in the elements of $U$; this is the content of the next proposition. The length of the required word can be bounded in terms of the cardinality of $\G^{(0)}$ alone, or instead in terms of $\dim(\G)$ and the cardinality of $U$. We will see in \Cref{sec:sharpness} that these two bounds cannot be qualitatively improved.

\begin{prop}\label{prop:full_support}
If $U\sq\mc{G}_{\G}$ is a subset, there exists an element $g\in U^n$ such that
\[ {\rm acon}\big(\supp(U)\big)\sq\stsupp(g) , \] 
where $1\leqslant n\leqslant N$, setting $k=\min\{|U|^2,|\G^{(0)}|\}$ and defining
\[ \begin{cases} N:= (6\dim(\G)+5)^k, & \text{if the vertex groups of $\mc{G}_{\G}$ have no $2$--torsion;} \\ N:= \big(2^{\dim(\G)+3}\dim(\G)\big)^k, & \text{in general.} \end{cases} \]
\end{prop}
\begin{proof}
To begin with, \Cref{lem:U^2_covers_support}(1) yields elements $u_1,\dots,u_k\in U^2$ such that the sets ${\rm acon}(\stsupp(u_i))$ cover ${\rm acon}(\supp(U))$, where we can take $k\leqslant \min\{|U|^2,|\G^{(0)}|\}$. Set $\kappa:=2^{\dim(\G)+3}\dim(\G)$ (or $\kappa:=6\dim(\G)+5$ if there is no $2$--torsion). We now apply \Cref{lem:combining_supp} and \Cref{rmk:ugly_bound} to $u_1$ and $u_2$, obtaining an element $g_2=u_1^mu_2^n$ with $m,n\leqslant\kappa$ and 
\begin{align*} 
{\rm acon}\big(\stsupp(u_1)\big)\cup{\rm acon}\big(\stsupp(u_2)\big) &\sq {\rm acon}\big(\stsupp(u_1)\cup\stsupp(u_2)\big) \\
&\sq {\rm acon}(\stsupp(g_2)) .
\end{align*}
In particular, we have $g_2\in U^{4\kappa}$. We can now again apply \Cref{lem:combining_supp}  to $g_2$ and $u_3$, obtaining an element $g_3$, then again to $g_3$ and $u_4$, and so on. At each step, we have an element $g_i\in U^{N_i}$ with 
\[ \bigcup_{j\leqslant i} {\rm acon}\big(\stsupp(u_1)\big) \sq {\rm acon}\big(\stsupp(g_i)\big) \]
 and $N_i=\kappa(N_{i-1}+2)$. Thus 
 \[N_i=4\kappa^{i-1}+2(\kappa^{i-2}+\dots+\kappa)\leqslant 5\kappa^{i-1},\] 
 where the inequality uses that $\kappa\geqslant 4$. At the end, we have ${\rm acon}\big(\supp(U)\big)\sq \stsupp(g)$ and $g_k\in U^{N_k}$ with $N_k\leqslant\kappa^{k-1}$, as required. 
\end{proof}

In the case that ${\rm acon}(\supp(U))=\supp(U)$, we get equality with $\stsupp(g)$. In particular, this holds whenever $\supp(U)$ is neither a single vertex nor a join, which gives us an effective version of Theorem 6.16 in \cite{Minasyan2015}. We can also show that this holds without any restrictions on $U$ when $\mc{G}_{\G}$ is torsion free:

\begin{cor}
\label{cor:torsion_free_full_support}
    Suppose $\mc{G}_{\G}$ is torsion free. If $U\sq\mc{G}_{\G}$ is a subset, there exists an element $g\in U^n$ such that
\[ \supp(U)=\stsupp(g)=\supp(g) , \] 
where $1\leqslant n\leqslant (7\dim(\G)+5)^{k'}$, setting $k'=\dim(\G)+\min\{|U|^2,|\G^{(0)}|\}$.
\end{cor}

\begin{proof}
    Write $\supp(U)=\{v_1\}\ast\dots\ast\{v_m\}\ast{\rm acon}(\supp(U))$, where the symbol $\ast$ denotes graph joins. As observed in \Cref{lem:U^2_covers_support}(2), the fact that the $v_i$ are cone vertices of $\supp(U)$ implies that there exist elements $g_i\in U$ with $v_i\in\supp(g)$.

    Set $d:=\dim(\G)$. Arguing by induction on $i$, we will construct elements $h_i\in U^{n_i}$ with $\{v_1,\dots,v_i\}\cup{\rm acon}(\supp(U))\sq\stsupp(h_i)$ so that $n_{i+1}\leq (7d+5)(n_i+1)$. For the base step, we observe that \Cref{prop:full_support} yields the existence of the element $h_0\in U^{n_0}$ with exponent $n_0=(6d+5)^k$ and $k=\min\{|U|,|\G^{(0)}|\}$.

    For the inductive step, suppose that we have found $h_i$ and we want to define $h_{i+1}:=g_{i+1}^ah_i^b$ for a suitable choice of $a,b>0$. If $a,b$ satisfy the requirements of \Cref{lem:combining_supp}, we have ${\rm acon}(\supp(U))\sq\stsupp(h_{i+1})$. For each $1\leq j\leq i+1$, we also want to ensure that $v_j\in\supp(h_{i+1})$ (which implies that $v_j\in\stsupp(h_{i+1})$ since $\mc{G}_{\G}$ is torsion-free). The latter amounts to the ratio $a/b$ missing at most $i+1$ rational numbers, as follows from the following claim applied to the projections of $g_{i+1}$ and $h_i$ to the vertex groups $G_{v_j}$. We omit the straightforward proof of the claim.

    \smallskip
    {\bf Claim.} \emph{Let $H$ be a group and let $x,y\in H$ be elements. The set $\Omega:=\{(m,n)\in\Z^2\mid x^my^n=1\}$ is a subgroup of $\Z^2$. If at least one among $x$ and $y$ has infinite order in $H$, then $\Omega$ has infinite index in $\Z^2$ and thus either $\Omega=\{1\}$ or $\Omega\cong\Z$.}

    \smallskip
    Summing up, we can define $h_{i+1}=g_{i+1}^ah_i^b$ provided that $a,b\geq 4d+5$, that $a$ and $b$ each miss sets $\Phi(g_{i+1})$ and $\Phi(h_i)$ of cardinality $\leq d$, and that the ratio $a/b$ misses a set of rational numbers of cardinality $\leq d+i+1\leq 2d$. Arguing as in \Cref{rmk:ugly_bound}, we can find such integers $a,b$ with $b\leq 5d+5$ and $a\leq 7d+5$. Since we had $h_i\in U^{n_i}$, it follows that $h_{i+1}\in U^{n_{i+1}}$ with $n_{i+1}\leq (5d+5)n_i+ (7d+5)\leq (7d+5)(n_i+1)$, as required.

    At the end of the induction, we obtain an element $h_m\in U^{n_m}$ with $\stsupp(h_m)=\supp(U)$ and $n_m\leq (7d+5)^{k+d}$, completing the proof.
\end{proof}

In particular, the above holds for right-angled Artin groups, so this generalises Theorem 4.2.11 in \cite{Kerr2021a}.

\begin{rem}
    If $\mc{G}_{\G}$ is not torsion free, then there is no guarantee that the stronger statement given in \Cref{cor:torsion_free_full_support} holds. For example, considering the subset $U=\{(1,1,0),(1,0,1)\}\subset\Z_2\times\Z_2\times\Z_2$, no element $g\in\langle U\rangle$ satisfies $\supp(g)=\supp(U)$.
\end{rem}

We are now ready to prove Theorem~\ref{thmintro:short_loxodromics}, whose statement we recall. Note that the integer $N$ provided by \Cref{prop:full_support} is bounded above both by a function of $\dim(\G)$ and $|U|$, and by a function of $|\G^{(0)}|$ alone.

\begin{thm}\label{thm:short_loxodromics}
Consider a subset $U\sq\mc{G}_{\G}$ and define $N$ as in \Cref{prop:full_support}.
    \begin{enumerate}
        \item If the subgroup $\langle U\rangle$ contains a regular element, then there exists an integer $1\leqslant n\leqslant N$ such that the product $U^n$ contains a regular element.
        \item If the subgroup $\langle U\rangle$ contains a strongly irreducible element, then there exists $1\leqslant n\leqslant N$ such that $U^n$ contains a strongly irreducible element.
    \end{enumerate}
\end{thm}

\begin{proof}
Let $C_U$ be the set of cone vertices of $\supp(U)$; we can write $\supp(U)$ as the join of the clique $C_U$ and the aconical part ${\rm acon}(\supp(U))$. Note that, in both cases of the theorem, the subgroup $\langle U\rangle$ contains a regular element $x$, namely an element such that $\stsupp(x)$ is not a single vertex and is not contained in a join of subgraphs of $\G$ generating infinite parabolic subgroups (a ``large join''). This implies that the clique $C_U$ has finite vertex groups, and that the aconical part ${\rm acon}(\supp(U))$ has at least two vertices, does not split as a join, and is not contained in a large join.

Any element $g\in\langle U\rangle$ such that ${\rm acon}(\supp(U))\sq\stsupp(g)$ is regular. If $\langle U\rangle$ contains a strongly irreducible element, then $C_U=\emptyset$ and ${\rm acon}(\supp(U))$ is not contained in joins of any kind. In this case, any element $g\in\langle U\rangle$ such that ${\rm acon}(\supp(U))\sq\stsupp(g)$ is strongly irreducible. 

In conclusion, for both parts of the theorem, it suffices to show that there exists an element $g\in U^N$ with ${\rm acon}(\supp(U))\sq\stsupp(g)$. The existence of such an element is precisely the content of \Cref{prop:full_support}.
\end{proof}

One case where it is clear that we have strongly irreducible elements is when $\supp(U)$ is neither a vertex nor lies in a join (see \cite[Theorem 6.16]{Minasyan2015}, for example). The following can also be deduced directly from \Cref{prop:full_support}, and will be needed for various growth results in \Cref{sec:growth}.

\begin{cor}
\label{ShortStrongIrreducible}
    If $U\sq\mc{G}_{\G}$ is a subset such that $\supp(U)$ is not a vertex and does not lie in a join, there exists a strongly irreducible element $g\in U^n$,
where $1\leqslant n\leqslant N$, and $N$ is as in \Cref{prop:full_support}.
\end{cor}

\begin{rem}
    Suppose $\mc{G}_{\G}$ is finitely generated, and $\G$ is neither a vertex nor a join. In this case, we obtain $g\in\mc{G}_{\G}$ such that $\supp(g)=\G$, so not only does $g$ act loxodromically on the contact graph of $\mc{G}_{\G}$, it also acts loxodromically in the acylindrical action of $\mc{G}_{\G}$ on $Cay(\mc{G}_{\G},\bigcup_{\Lambda\subsetneq \G}\mc{G}_{\Lambda})$. This is the hyperbolic space associated to the maximal domain in a relatively hierarchically hyperbolic structure on $\mc{G}_{\G}$, as shown in \cite{Berlyne2022}.
\end{rem}

\subsection{The bounded creasing property}\label{subsec:bounded_creasing}

The proof of \Cref{thmintro:short_loxodromics} in the previous subsection relied on the bounded creasing property described in \Cref{prop:bounded_creasing}. The current subsection proves this result.

The main ingredient is the following lemma proved in \cite{Casals2019,Casals2021}; we briefly recall parts of the argument, as it requires combining various results from these papers and working through some involved notation. Note that part~(2) of this lemma is exactly part~(2) of \Cref{prop:bounded_creasing}.

\begin{lem}\label{lem:fellow_irreducibles(new)}
Let $g\in\mc{G}_{\G}$ be loxodromic in $\T_v$ for some $v\in\G$.
\begin{enumerate}
\item If $h\in\mc{G}_{\G}$ fixes an arc of length $>2\dim(\G)\cdot\tau_v(g)$ of the axis $A_v(g)$, then $hA_v(g)=A_v(g)$. If $g$ is irreducible, then $[g,h]=1$.
\item Let $h\in\mc{G}_{\G}$ be loxodromic in $\T_v$. Let $\overline g,\overline h$ be the irreducible components of $g,h$ that are loxodromic in $\T_v$. If we have 
\[ \ell\big(A_v(g)\cap A_v(h)\big)> (2\dim(\G)+2)\cdot \max\{\tau_v(g),\tau_v(h)\}, \] 
then $\langle \overline g,\overline h\rangle\cong\Z$. 
\end{enumerate}
\end{lem}
\begin{proof}
If $P\leq\mc{G}_{\G}$ is a parabolic subgroup and $x\in\mc{G}_{\G}$ is any element, then the intersection $\bigcap_{0\leq k\leq n} x^kPx^{-k}$ is constant for $n\geq 2\dim(\G)$; this follows from \cite[Theorem~3.10 and Remark~2.3]{Casals2021}. Now, if $E$ is the stabiliser of any edge of the axis $A_v(g)$, then $E$ is parabolic and $\bigcap_{0\leq k\leq n} g^kEg^{-k}$ is the stabiliser of an arc of $A_v(g)$ of length $n\tau_v(g)+1$; in fact, the stabiliser of any such arc can be written as such an intersection, for a suitable choice of $E$. This shows that the stabiliser of any arc of $A_v(g)$ of length $>2\dim(\G)\tau_v(g)$ actually fixes $A_v(g)$ pointwise, proving the first half of part~(1).

If $P$ is the pointwise stabiliser of $A_v(g)$, then $P$ is parabolic by \Cref{lem:parabolic_intersections} and its normaliser splits as $N(P)=P\x P^{\perp}$ by \Cref{lem:parabolic_normaliser}. We have $g\in N(P)$ because $g$ preserves $A_v(g)$. If $g$ is irreducible, then either $g\in P$ or $g\in P^{\perp}$, and the former is barred by the fact that $g$ translates non-trivially along $A_v(g)$. Hence $g\in P^{\perp}$ and, if $h$ lies in $P$, we have $[g,h]=1$, completing the proof of part~(1).

Regarding part~(2), observe that the commutator $[\overline g,\overline h]$ fixes an arc of $A_v(g)\cap A_v(h)$ of length $>2\dim(\G)\cdot \max\{\tau_v(g),\tau_v(h)\}$. If we had $A_v(g)\neq A_v(h)$, then $[\overline g,\overline h]$ would not preserve $A_v(g)$, contradicting part~(1) (we learnt this neat trick from \cite[Lemma~2.10]{Casals2019}). Thus, we must have $A_v(g)=A_v(h)$. Let $P$ again be the pointwise stabiliser of this axis; as above, we have $N(P)=P\x P^{\perp}$ and $\overline g,\overline h\in P^{\perp}$. There is a homomorphism $T\colon\langle \overline g,\overline h\rangle\ra\Z$ given by (signed) translation lengths along the shared axis $A_v(g)= A_v(h)$. Since $\ker T$ fixes this axis pointwise and $\overline g,\overline h\in P^{\perp}$, it follows that $\ker T\leq P\cap P^{\perp}=\{1\}$, hence $T$ provides the required isomorphism between $\langle\overline g,\overline h\rangle$ and (a subgroup of) $\Z$.
\end{proof}

In order to deduce the rest of \Cref{prop:bounded_creasing} from \Cref{lem:fellow_irreducibles(new)}, we need part~(2) of the next result. Part~(1) should help clarify the type of situation that we need to worry about.

\begin{lem}\label{lem:inverse_conjugacy}
Consider an element $g\in\mc{G}_{\G}$.
\begin{enumerate}
\item If $\supp(g)$ does not have any cone vertices, then $g$ is $\mc{G}_{\G}$--conjugate to its inverse if and only if it is a product of two order--$2$ elements.
\item If $\supp(g)$ is not a clique and $h\in\mc{G}_{\G}$ is an element with $hgh^{-1}=g^{-1}$, then at least one of the irreducible components of $h$ has order $2$. 
\end{enumerate}
\end{lem}
\begin{proof}
Note that an element of $\mc{G}_{\G}$ conjugates $g$ to its inverse if and only if it conjugates each irreducible component of $g$ to its inverse; this follows from the uniqueness of irreducible components and the fact that conjugations preserve essential supports. Thus, it suffices to prove both parts of the lemma under the assumption that $g$ is irreducible and $\supp(g)$ is not a single vertex. We make this assumption in the rest of the proof. Consider some element $h\in\mc{G}_{\G}$ such that $hgh^{-1}=g^{-1}$ throughout.

We first prove part~(2). Pick some $v\in\supp(g)$. Since $g$ is irreducible and $\supp(g)$ is not a single vertex, $g$ is loxodromic in $\T_v$. Since $hgh^{-1}=g^{-1}$, the element $h$ preserves the axis $A_v(g)$ swapping its two ends. Letting $P\leqslant \mc{G}_{\G}$ be the pointwise stabiliser of $A_v(g)$, we have $h\in N(P)\setminus P$ and $h^2\in P$. Recalling that $N(P)$ splits as $P\x P^{\perp}$, we can write $h=ab$ where $a\in P$ and $b$ is an order--$2$ element of $P^{\perp}$. This shows that $h$ has an irreducible component of order $2$ (any of the irreducible components of $b$), which proves part~(2) of the lemma.

We now prove part~(1). In the above notation, \Cref{lem:fellow_irreducibles(new)}(1) shows that $P$ commutes with $g$, so we have $g^{-1}=hgh^{-1}=bgb^{-1}$. Now, we can write
\[ g=b\cdot (bg), \]
where both $b$ and $bg$ have order $2$, since
\[ (bg)^2=bgb\cdot g=bgb^{-1}\cdot g=g^{-1}g=1. \]
Conversely, it is clear that any product of order--$2$ elements is conjugate to its inverse: if $x^2=y^2=1$, then $yx=x(xy)x^{-1}$.
\end{proof}

We are finally ready to prove the bounded creasing property.

\begin{proof}[Proof of \Cref{prop:bounded_creasing}]
Part~(2) of the proposition is \Cref{lem:fellow_irreducibles(new)}(2) above. In order to prove part~(1), consider an element $g\in\mc{G}_{\G}$ that is loxodromic in $\T_v$ and an element $h\in\mc{G}_{\G}$ without order--$2$ irreducible components. We need to show that $h$ does not crease an arc of length $> (2\dim(\G)+2)\tau_v(g)$ of the axis $A_v(g)$.

Suppose for the sake of contradiction that $h$ does crease such an arc $\alpha\sq A_v(g)$. Without loss of generality, we can assume that $g$ is irreducible (recall \Cref{rmk:tree_vs_components}). Then, since 
\[ \alpha\sq A_v(g)\cap A_v(hgh^{-1}) \quad \text{and} \quad \tau_v(g)=\tau_v(hgh^{-1}), \] \Cref{lem:fellow_irreducibles(new)}(2) implies that $\langle g,hgh^{-1}\rangle\cong\Z$. Since $g$ and $hgh^{-1}$ have the same translation length and translate in opposite directions along the intersection of their axes, we conclude that $hgh^{-1}=g^{-1}$. Finally, since $h$ does not have any order--$2$ irreducible components by assumption, Lemma~\ref{lem:inverse_conjugacy}(2) yields the required contradiction.
\end{proof}

\subsection{Applications to special groups}
\label{sub:special}

This section is devoted to the proof of \Cref{corintro:special} and its analogue for fundamental groups of non-compact special cube complexes. We thus restrict to the setting of a right-angled Artin group $A_{\G}$ and its subgroups.

Recall that an element $g$ of a group $G$ is said to be a \emph{generalised loxodromic} if there exists an acylindrical action of $G$ on a hyperbolic space such that $g$ is loxodromic; see \cite[Theorem~1.4]{Osin2016} for various equivalent characterisations of such elements. The group $G$ admits generalised loxodromics if and only if it is either acylindrically hyperbolic or virtually cyclic. When $G$ is finitely generated, every generalised loxodromic is a Morse element in $G$, with respect to the word metric induced by any finite generating set \cite{Sisto-MathZ}.

\begin{prop}\label{prop:ah_equivalence}
    There exists a constant $N=N(\G)$ with the following property. For any nontrivial subgroup $G\leq A_{\G}$, the following are equivalent:
    \begin{enumerate}
        \item $G$ is acylindrically hyperbolic or infinite cyclic;
        \item $G$ does not have a normal subgroup that splits as the direct product of two infinite groups;
        \item for any generating set $S\sq G$, a product $S^n$ with $1\leq n\leq N$ contains a generalised loxodromic for $G$.
    \end{enumerate}
    Alternatively, the constant $N$ can be taken to depend on $\dim(\G)$ and $|S|$.
\end{prop}
\begin{proof}
    The implications $(3)\Ra(1)\Ra(2)$ are clear (see e.g.\ \cite[Corollary~1.5]{Osin2016} for the latter). Let us show that $(2)\Ra(3)$. Assume without loss of generality that $G$ is not cyclic, so that $\supp(G)\sq\G$ has cardinality $\geq 2$. Also suppose that $|\G^{(0)}|$ is minimal among all embeddings of $G$ into right-angled Artin groups, which in particular implies that $\supp(G)=\G$. 
    
    If $\G$ is not a join, \Cref{ShortStrongIrreducible} gives a product $S^n$, with $n$ as in the statement of the proposition, and an element $g\in S^n$ that is loxodromic for the action of $A_{\G}$ on its contact graph. The restriction of this action to $G$ witnesses the fact that $g$ is generalised loxodromic in $G$. 

    If instead $\G$ is a join, then $A_{\G}=A_{\G_1}\x A_{\G_2}$ for two (nonempty) subgraphs $\G_1,\G_2\sq\G$. If one of the intersections $G_i:=G\cap A_{\G_i}$ were to be trivial, say $G_1=\{1\}$, then $G$ would project injectively to a subgroup of $A_{\G_2}$, violating the minimality assumption on $\G$. Thus, both $G_1$ and $G_2$ are infinite, and we can consider the product set $G_1\cdot G_2$, which is a normal subgroup of $G$ because the $G_i$ are normal in $G$. In fact, since the $G_i$ commute with each other and $G_1\cap G_2=\{1\}$, we have $G_1\cdot G_2\cong G_1\x G_2$, so this is the required normal subgroup of $G$ splitting as a direct product.
\end{proof}

A cube complex is said to be \emph{special} if it admits a locally isometric immersion into the Salvetti complex of a right-angled Artin group \cite{Haglund-Wise2008}. The following implies \Cref{corintro:special} from the introduction.

\begin{cor}\label{cor:ah_equivalence2}
    There exists $N=N(h)$ with the following property. If $G\neq\{1\}$ is the fundamental group of a special cube complex $C$ with $h$ hyperplanes, then the following are equivalent:
    \begin{enumerate}
        \item $G$ is acylindrically hyperbolic or infinite cyclic;
        \item $G$ does not have a normal subgroup that splits as the direct product of two infinite groups;
        \item for any generating set $S\sq G$, a product $S^n$ with $1\leq n\leq N$ contains a generalised loxodromic for $G$.
    \end{enumerate}
    If $C$ is compact, then the following are equivalent for any subset $U\sq G\setminus\{1\}$:
    \begin{enumerate}
        \item[(a)] $\langle U\rangle$ is not virtually contained in a direct product of two infinite subgroups of $G$;
        \item[(b)] a product $U^n$ with $1\leq n\leq N$ contains a generalised loxodromic for $G$.
    \end{enumerate}
    In both parts of the corollary, the constant $N$ can be alternatively taken to depend on $\dim(C)$ and the cardinality of $S$ or $U$.
\end{cor}

\begin{proof}
    The first part of the corollary is just a restatement of \Cref{prop:ah_equivalence}, as $G$ embeds in a right-angled Artin group $A_{\G}$ with $|\G^{(0)}|=h$; see \cite[Section~4]{Haglund-Wise2008}. 

    For the second half, suppose that $C$ is compact, and similarly embed $G$ in a right-angled Artin group $A_{\G}$ so that $G$ acts cocompactly on a convex subcomplex $\widetilde C$ of the universal cover of the Salvetti complex of $A_{\G}$. In particular, this implies that for any product of parabolic subgroups $P_1\x P_2$ within $A_{\G}$, the intersection $G\cap (P_1\x P_2)$ contains the product $(G\cap P_1)\x (G\cap P_2)$ as a subgroup of finite index (for instance, this follows by combining Lemmas~2.8 and~2.4 in \cite{Fio-spec}).
    
    The implication $(b)\Ra(a)$ is clear. Conversely, let us assume that $\langle U\rangle$ is not virtually contained in a product within $G$ and construct a short generalised loxodromic in $\langle U\rangle$. An element $g\in G\setminus\{1\}$ is generalised loxodromic in $G$ if and only if the centraliser $Z_G(g)$ is cyclic; indeed, elements with cyclic centraliser are contracting in $\widetilde C$ \cite[Theorem~1.3]{Genevois-rank1}, and the latter are generalised loxodromics by \cite[Theorem~H]{Bestvina2015} and \cite[Theorem~1.2]{Osin2016}. Now, \Cref{cor:torsion_free_full_support} yields an integer $n$ as in the statement of the corollary such that the product $U^n$ contains an element $g_0$ with $\supp(g_0)=\supp(U)$. We conclude the proof by showing that $Z_G(g_0)\cong\Z$. 
    
    The support $\supp(U)$ is not a join. Otherwise, $U$ would be contained in a product $P=P_1\x P_2$ of parabolic subgroups of $A_{\G}$ with $\supp(U)=\supp(P)$. As observed above, the product $(G\cap P_1)\x (G\cap P_2)$ would have finite index in $G\cap P$, and a power of $g_0$ would lie in this product. Since $\supp(g_0)=\supp(U)$, this shows that both $G\cap P_1$ and $G\cap P_2$ would be nonempty, violating the assumption that $\langle U\rangle$ is not virtually contained in a product. 
    
    Now, since $\supp(g_0)$ is not a join, $g_0$ is irreducible and we have $Z_{A_{\G}}(g_0)=\langle g_0'\rangle\x Q$, where $g_0$ is a power of the element $g_0'\in A_{\G}$, and $Q$ is a parabolic subgroup of $A_{\G}$ such that $\G$ contains the join $\supp(g_0)\ast\supp(Q)$ \cite[Section~3]{Servatius1989/2}. In particular, $G\cap Q$ commutes with $U$, and $Z_G(g_0)$ virtually splits as $\langle g_0\rangle\x (G\cap Q)$. Since $\langle U\rangle$ is not contained in a product of infinite groups, we must have $G\cap Q=\{1\}$, so $Z_G(g_0)$ is virtually $\langle g_0\rangle$. All virtually cyclic subgroups of $A_{\G}$ are cyclic, hence $Z_G(g_0)\cong\Z$ and $g_0$ is the required generalised loxodromic, completing the proof.
\end{proof}

\section{Sharpness of the main theorem}
\label{sec:sharpness}

In this section we present two examples showing that \Cref{thmintro:short_loxodromics} is false if the power $N$ is only allowed to depend on the clique size $\dim(\G)$ (\Cref{ex:need_cardinality}), and also if it is only allowed to depend on the cardinality $|U|$ (\Cref{prop:need_dimension}). This is already the case in right-angled Artin groups, even if the set $U$ is symmetric.

Recall that, in right-angled Artin groups, regular and strongly irreducible elements coincide: they are simply contracting elements (for the standard word metric) not conjugate to a power of a standard generator (the latter requirement is vacuous if $\G$ is connected, and is not a single vertex). For this reason, we will refer to such elements as contracting elements in this section.

\begin{exm}\label{ex:need_cardinality}
This example is to show that dependence on $|U|$ in \Cref{thmintro:short_loxodromics} cannot be dropped, not even for $2$--dimensional right-angled Artin groups. Let the graph $\G$ have $2m$ vertices labelled $x_1,\dots,x_m,y_1,\dots,y_m$ and edges $[x_i,y_j]$ whenever $i\neq j$. The case $m=3$ is depicted in Figure~\ref{fig:bipartite}. Note that $\G$ does not contain any cliques with $3$ vertices; in other words, the Salvetti complex is $2$--dimensional. Consider the symmetric subset $U=\{x_1^{\pm},\dots,x_m^{\pm}\}\sq A_{\G}$, which generates a subgroup containing the contracting element $x_1x_2\cdots x_m$. At the same time, for each $1\leqslant k<m$ the product $U^k$ does not contain any contracting elements: for every element $g\in U^k$, there exists some $x_i$ that does not appear in any reduced word representing $g$, which implies that $y_i$ commutes with $g$.
\end{exm}

\begin{figure}[ht]
\centering
\begin{tikzpicture}
\draw[fill] (0,0) circle [radius=0.07cm];
\draw[fill] (0,1) circle [radius=0.07cm];
\draw[fill] (0,2) circle [radius=0.07cm];
\draw[fill] (1.5,0) circle [radius=0.07cm];
\draw[fill] (1.5,1) circle [radius=0.07cm];
\draw[fill] (1.5,2) circle [radius=0.07cm];
\draw[thick] (0,0) -- (1.5,1);
\draw[thick] (0,0) -- (1.5,2);
\draw[thick] (0,1) -- (1.5,0);
\draw[thick] (0,1) -- (1.5,2);
\draw[thick] (0,2) -- (1.5,1);
\draw[thick] (0,2) -- (1.5,0);
\node [left] at (0,0) {$x_1$};
\node [left] at (0,1) {$x_2$};
\node [left] at (0,2) {$x_3$};
\node [right] at (1.5,0) {$y_1$};
\node [right] at (1.5,1) {$y_2$};
\node [right] at (1.5,2) {$y_3$};
\node [right] at (-2,1) {\Large $\G\ =$};
\end{tikzpicture}
\caption{The case $m=3$ of \Cref{ex:need_cardinality}.}
\label{fig:bipartite}
\end{figure}

We now show that the dependence on $\dim(\G)$ also cannot be dropped from \Cref{thmintro:short_loxodromics}; this construction is a little more involved and we discuss it in \Cref{prop:need_dimension} below. First, however, it is interesting to observe that there are much simpler examples showing that the powers needed to combine the essential supports of two elements truly depend on the $\dim(\G)$ in general.

\begin{exm}\label{ex:big_abelian}
For every $N\geqslant 1$, there exist two elements $g,h\in\Z^{(2N+1)^2-1}$ such that the subgroup $\langle g,h\rangle$ has full support in $\Z^{(2N+1)^2-1}$, but, for all integers $-N\leqslant m,n\leqslant N$, the product $g^mh^n$ does not (i.e.\ it does not require all standard generators in order to be written).

Let $(a_1,b_1),(a_2,b_2),\dots$ be an enumeration of all non-zero pairs with integer coordinates in $[-N,N]\x[-N,N]$; there are $(2N+1)^2-1$ such pairs. Let $x_1,\dots,x_{(2N+1)^2-1}$ be the standard generators of $\Z^{(2N+1)^2-1}$. It suffices to consider the elements $g=x_1^{b_1}x_2^{b_2}\cdots $ and $h=x_1^{-a_1}x_2^{-a_2}\cdots $. Then, for all $-N\leqslant m,n\leqslant N$, there exists some index $i$ such that $(a_i,b_i)=(m,n)$, from which it follows that $x_i$ does not appear in $g^mh^n$.
\end{exm}

Of course, in the previous example there are no contracting elements. The next result shows that, nevertheless, generating a contracting element will require powers depending on $\dim(\G)$ in general.

\begin{prop}\label{prop:need_dimension}
For every $N\geqslant 1$, there exist a right-angled Artin group $A_{\G}$ and two elements $g,h\in A_{\G}$ such that both the following hold:
\begin{enumerate}
\item the subgroup $\langle g,h\rangle$ contains a contracting element in $A_{\G}$;
\item the product $\{1,g,h,g^{-1},h^{-1}\}^N$ does not contain any contracting elements in $A_{\G}$.
\end{enumerate}
\end{prop}
\begin{proof}
First, observe that it suffices to construct a graph $\Lambda$ that is not a join and two elements $g,h\in A_{\Lambda}$ such that $\langle g,h\rangle$ contains an element with $\supp(\cdot)=\Lambda$, while no element of the product $\{1,g,h,g^{-1},h^{-1}\}^N$ has $\supp(\cdot)=\Lambda$. Once we have this, we can consider the graph $\G$ that has two vertices $v_1,v_2$ for each vertex $v\in\Lambda$, with the vertices in $\{v_1\mid v\in\Lambda\}$ spanning a copy of $\Lambda$, and each vertex $v_2$ connected exactly to the vertices in $\{w_1 \mid w\in\Lambda\setminus\{v\}\}$. Considering the previously constructed elements $g,h$ in the copy of $A_{\Lambda}$ within $A_{\G}$ yields the proposition (an element of $A_{\Lambda}$ is contracting in $A_{\G}$ if and only if its essential support is the whole $\Lambda$).

Thus, the rest of the proof is devoted to the construction of $\Lambda$. In fact, it is simpler to describe the ``opposite graph'' of $\Lambda$, which we denote by $\Lambda^o$. This is the graph with the same vertex set as $\Lambda$ and with two vertices connected by an edge if and only if they are \emph{not} connected by an edge in $\Lambda$. 

The graph we will require is depicted in Figure~\ref{fig:co-star}, where $s,t$ are very large integers (related to $N$). Explicitly, $\Lambda^o$ is a tree with a central vertex $y$ and $s$ branches of length $t$; the vertex of the $i$--th branch at distance $j$ from $y$ is denoted $x_{ij}$. 

\begin{figure}[ht]
\centering
\begin{tikzpicture}
\draw[fill] (0,0) circle [radius=0.07cm];
\draw[fill] (-.7,0) circle [radius=0.07cm];
\draw[fill] (-1.4,0) circle [radius=0.07cm];
\draw[fill] (-3.5,0) circle [radius=0.07cm];
\draw[fill] ({0.7*cos(140)},{0.7*sin(140)}) circle [radius=0.07cm];
\draw[fill] ({1.4*cos(140)},{1.4*sin(140)}) circle [radius=0.07cm];
\draw[fill] ({3.5*cos(140)},{3.5*sin(140)}) circle [radius=0.07cm];
\draw[fill] ({0.7*cos(100)},{0.7*sin(100)}) circle [radius=0.07cm];
\draw[fill] ({1.4*cos(100)},{1.4*sin(100)}) circle [radius=0.07cm];
\draw[fill] ({3.5*cos(100)},{3.5*sin(100)}) circle [radius=0.07cm];
\draw[fill] (.7,0) circle [radius=0.07cm];
\draw[fill] (1.5,0) circle [radius=0.07cm];
\draw[fill] (3.5,0) circle [radius=0.07cm];
\draw[thick] (0,0) -- (-3.5,0);
\draw[thick] (0,0) -- ({3.5*cos(140)},{3.5*sin(140)});
\draw[thick] (0,0) -- ({3.5*cos(100)},{3.5*sin(100)});
\draw[thick] (0,0) -- (3.5,0);
\draw[fill] ({2.5*cos(45)},{2.5*sin(45)}) circle [radius=0.02cm];
\draw[fill] ({2.5*cos(50)},{2.5*sin(50)}) circle [radius=0.02cm];
\draw[fill] ({2.5*cos(55)},{2.5*sin(55)}) circle [radius=0.02cm];
\draw[fill] (-2.3,-.2) circle [radius=0.02cm];
\draw[fill] (-2.45,-.2) circle [radius=0.02cm];
\draw[fill] (-2.6,-.2) circle [radius=0.02cm];
\draw[fill] (2.3,-.2) circle [radius=0.02cm];
\draw[fill] (2.45,-.2) circle [radius=0.02cm];
\draw[fill] (2.6,-.2) circle [radius=0.02cm];
\node [below] at (0,0) {$y$};
\node [below] at (-.7,0) {$x_{11}$};
\node [below] at (-1.4,0) {$x_{12}$};
\node [below] at (-3.5,0) {$x_{1t}$};
\node [below] at (.7,0) {$x_{s1}$};
\node [below] at (1.4,0) {$x_{s2}$};
\node [below] at (3.5,0) {$x_{st}$};
\node [above] at ({0.7*cos(140)},{0.7*sin(140)}) {$x_{21}$};
\node [above] at ({1.4*cos(140)},{1.4*sin(140)}) {$x_{22}$};
\node [above] at ({3.5*cos(140)},{3.5*sin(140)}) {$x_{2t}$};
\node [right] at ({0.7*cos(100)},{0.7*sin(100)}) {$x_{31}$};
\node [right] at ({1.4*cos(100)},{1.4*sin(100)}) {$x_{32}$};
\node [right] at ({3.5*cos(100)},{3.5*sin(100)}) {$x_{3t}$};
\node [right] at (-6,1.25) {\Large $\Lambda^o\ =$};
\end{tikzpicture}
\caption{The opposite graph of $\Lambda$.}
\label{fig:co-star}
\end{figure}
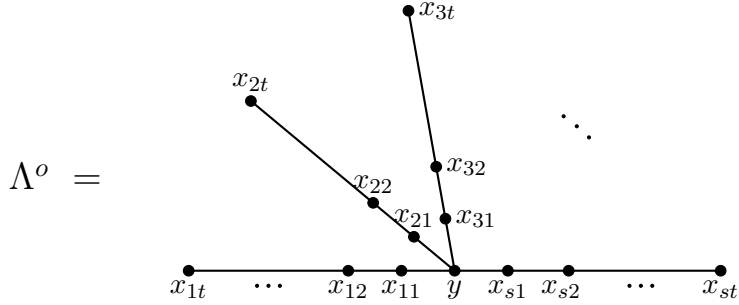

For each $1\leqslant i\leqslant s$, define 
\[ v_i:=x_{it} \qquad\text{and}\qquad u_i:=x_{it}\cdots x_{i3}\cdot x_{i2}\cdot x_{i1} , \]
which we can think of as ``representatives'' of the leaves and branches of $\Lambda^o$, respectively. We think of the $v_i$ and $u_i$ as elements of $A_{\Lambda}$. Let $T_i$ be the tree dual to the hyperplanes labelled by $v_i$ in the universal cover of the Salvetti complex $X_{\Lambda}$ (this corresponds to an HNN splitting of $A_{\G}$, so it is not quite one of the trees studied in \Cref{sect:graph_product_main}; see \Cref{rmk:HNN_tree}). Let $1_i\in T_i$ be the projection to $T_i$ of the identity vertex of $X_{\Lambda}$.

Given integers $a_1,\dots,a_s\geqslant 0$, we will consider elements of the following form:
\begin{align*} 
\alpha:=u_1^{a_1}\cdot u_2^{a_2}\cdots u_s^{a_s}\cdot y .
\end{align*}
Note that $u_1,\dots,u_s$ pairwise commute in $A_{\Lambda}$, while none of them commutes with $y$. In addition, $u_i$ is loxodromic in $T_i$ with translation length $1$, while $\alpha$ is loxodromic in each $T_i$ with translation length $a_i$ (if $a_i\neq 0$). Since $u_i$ and $\alpha$ are cyclically reduced, the vertex $1_i$ lies on both of their axes in $T_i$. 

The key observation is that the axes of $\alpha$ and $u_i$ in $T_i$ share a long arc starting at $1_i$.

\smallskip
{\bf Claim.} \emph{The axes of $\alpha$ and $u_i$ in $T_i$ share an arc of length $\geqslant t$ starting at $1_i$ (and continuing in the direction of translation of both $\alpha$ and $u_i$). Moreover, the initial portion of length $t-1$ of this arc is fixed pointwise by the element $y\in A_{\Lambda}$.}

\smallskip
\emph{Proof of claim.} 
We first prove the slightly weaker fact that $\alpha^k\cdot 1_i=u_i^{ka_i}\cdot 1_i$ for every integer $k\geqslant 0$ such that $ka_i\leqslant t$. 

Consider the reduced word for $\alpha^k$ obtained by stringing one after the other $k$ copies of the word for $\alpha$ given above. Starting from this word, look for all occurrences of the letter $v_i=x_{i,t}$ and push each of these as far left as possible. Arguing this way, it is not hard to see that, for $ka_i\leqslant t$, we have
\[ \alpha^k\in x_{i,t}(x_{i,t-1}x_{i,t})(x_{i,t-2}x_{i,t-1}x_{i,t})\cdots (x_{i,t-ka_i+1}\cdots x_{i,t-1}x_{i,t})\cdot A_{\Lambda\setminus\{v_i\}}. \]
Exactly in the same way, the element $u_i^{ka_i}$ lies in the same left coset of $A_{\Lambda\setminus\{v_i\}}$. Since $A_{\Lambda\setminus\{v_i\}}$ is the stabiliser of the vertex $1_i\in T_i$, this proves that $\alpha^k\cdot 1_i=u_i^{ka_i}\cdot 1_i$. 

The full strength of the first part of the claim can now be obtained by an identical argument, showing that $\alpha^{k+1}$ and $u_i^t$ have initial subwords that contain the first $t$ occurrences of $x_{i,t}$ and only differ by right multiplication by an element of $A_{\Lambda\setminus\{v_i\}}$. 

Finally, the fact that the initial portion of length $t-1$ of the shared arc is fixed by $y$ follows from the fact that the element
\[ x_{i,t}(x_{i,t-1}x_{i,t})(x_{i,t-2}x_{i,t-1}x_{i,t})\cdots (x_{i,2}\cdots x_{i,t-1}x_{i,t}) \]
commutes with $y$. This completes the proof of the claim.
\hfill$\blacksquare$

\smallskip
Now, we can argue essentially as in Example~\ref{ex:big_abelian}. Take $s:=(N+1)(2N+1)-1$ and let $(a_i,b_i)$ be an enumeration of all non-zero pairs in $[0,N]\x[-N,N]$ with integer coordinates. Choose $t\geqslant 2N^2+1$. Consider:
\begin{align*}
\alpha&:=u_1^{a_1}\cdot u_2^{a_2}\cdots u_s^{a_s}\cdot y, & \beta&:=u_1^{b_1}\cdot u_2^{b_2}\cdots u_s^{b_s}.
\end{align*}
Note that, if $b_i\neq 0$, then the axis of $\beta$ in $T_i$ coincides with the axis of $u_i$ (unlike the axis of $\alpha$, which only shares a long portion with it); this is because the elements $u_j$ pairwise commute. Finally, we set:
\begin{align*}
g&:=(u_1u_2\cdots u_s)^{-N^2}\cdot\alpha\cdot (u_1u_2\cdots u_s)^{N^2}, & h&:=\beta.
\end{align*}
Since $(a_i,b_i)\neq (0,0)$ for every $i$, we have $\supp(g)\cup\supp(h)=\Lambda$.

By the claim, the axes of $\alpha$ and $u_i$ in $T_i$ share an arc of length $\geqslant 2N^2+1$ based at $1_i$ if $a_i\neq 0$; if instead $a_i=0$, then $\alpha$ fixes an initial arc of length $\geqslant 2N^2$ of the axis of $u_i$ (this follows from the second part of the claim and the fact that the $u_j$ pairwise commute). Since $g$ is a suitable conjugate of $\alpha$, it follows that the axes of $g$ and $h$ in $T_i$ coincide within the ball of radius $N^2$ around $1_i$ if $0\not\in\{a_i,b_i\}$; when instead $0\in\{a_i,b_i\}$, one among $g$ and $h$ fixes pointwise the portion of the axis of the other that lies in the $N^2$--ball around $1_i$. If $b_i>0$ and $a_i\neq 0$, then $g$ and $h$ translate in the same direction along the intersection of their axes; if $b_i<0$ and $a_i\neq 0$, they translate in opposite directions.

Now, consider an element $z\in\{1,g,g^{-1},h,h^{-1}\}^N$. We want to show that $z$ is elliptic in at least one of the trees $T_i$; equivalently, the essential support of $z$ does not contain $v_i$, which implies that it is a proper subgraph of $\Lambda$.

We can write $z=g^{\g_1}h^{\delta_1}\cdots g^{\g_k}h^{\delta_k}$ for integers $\g_i,\delta_i$ with $m:=\sum\g_i\in [-N,N]$ and $n:=\sum\delta_i\in [-N,N]$. Up to replacing $g$ with its inverse, we can assume that $m\geqslant 0$.

If $n\geqslant 0$, look at the tree $T_i$ for which we have $(a_i,b_i)=(n,-m)$. If $n<0$, look at the tree $T_i$ with $(a_i,b_i)=(-n,m)$. In both cases, the powers of $g$ appearing in $z$ are responsible for moving $1_i$ by a total distance of $|mn|$ along the shared portion of the axes of $g$ and $h$, while the powers of $h$ are responsible for moving $1_i$ by a total distance of $|mn|$ in the opposite direction. Since $|mn|\leqslant N^2$, the whole movement takes place within the arc shared by the axes of $g$ and $h$. In conclusion, this shows that $z$ fixes $1_i$, as required.

This completes the proof of the proposition.
\end{proof}

\section{Growth in graph products}
\label{sec:growth}

As a consequence of the work in \Cref{sect:graph_product_main}, we are able to obtain various results relating to the growth of graph products and their subgroups. First, \Cref{sec:well_ordered} quickly proves \Cref{corintro:well_ordered}, showing that the set of growth rates of certain subgroups of graph products is well-ordered. Then, \Cref{sub:product_set_growth} is concerned with product set growth and the proofs of \Cref{corintro:growth1} and \Cref{thmintro:growth_dichotomy}. We finish with \Cref{sec:TitsAlternative}, where we show that another corollary of \Cref{thm:short_loxodromics} is an effective Tits alternative for graph products.

As above, we will always consider graph products $\mc{G}_{\G}$ along a finite graph $\G$. 

\subsection{Well-ordered growth rates}
\label{sec:well_ordered}

A fairly immediate consequence of \Cref{thm:short_loxodromics} is that when the vertex groups of a graph product are equationally noetherian, the set of growth rates of many of its subgroups is well ordered.

For a finitely generated group $G$, with finite generating set $S$, we denote by $B_S(n)$ the ball of radius $n$, with respect to the word metric induced by $S$. The \emph{exponential growth rate of $G$ with respect to $S$} is $\omega(G,S)=\lim_{n\to\infty}|B_S(n)|^{\frac{1}{n}}$. Note that for this definition, it does not matter whether we pick $S$ to be symmetric or not, as $\omega(G,S)=\omega(G,S\cup S^{-1})$. We also recall that for a graph $\G$, the \emph{girth} of $\G$, denoted girth($\G$), is the length of the shortest cycle in $\G$.

The following is \Cref{corintro:well_ordered} in the introduction; we recall the statement for convenience.

\begin{thm}
\label{WellOrdered}
     Let $\mc{G}_{\G}$ be a graph product of equationally noetherian groups. The following set is well-ordered:
    \begin{align*}
        \{ \omega(H,S) \mid & \text{ $S\sq\mc{G}_{\G}$ finite, $H=\langle S\rangle$, $\supp(S)$ neither a single vertex nor a join, and}
     \\ & \ \emph{girth}(\supp(S))\geqslant 6 \}.
    \end{align*} 
\end{thm}

\begin{proof}
    Suppose that $\G'$ is an induced subgraph of $\G$ with $\text{girth}(\G')\geqslant 6$, and that $\G'$ is neither a single vertex nor a join. Let $S\subseteq \mc{G}_{\G'}$ be a finite set such that $\text{esupp}(S)=\G'$.
    
    Applying \Cref{ShortStrongIrreducible}, we can find $N=N(\G')\in\mathbb{N}$ such that there exists $n\leqslant N$ with $S^n$ containing a strongly irreducible element in $\mc{G}_{\G'}$. By \Cref{StronglyIrreducibleIsLoxodromic}, this element is loxodromic in the action of $\mc{G}_{\G'}$ on the contact graph $C(\mc{G}_{\G'})$. This action is acylindrical, and the contact graph is a quasi-tree \cite[Corollary C]{Valiunas2021}.
    
    By \cite[Theorem E]{Valiunas2021}, as $\text{girth}(\G')\geqslant 6$ and the vertex groups are equationally noetherian, we have that $\mc{G}_{\G'}$ is equationally noetherian.  By \cite[Theorem 6.1]{Fujiwara2021}, the following set is well-ordered:
    \begin{align*}
        \{ \omega(H,S) \mid \text{ $S\sq\mc{G}_{\G'}$ finite, $H=\langle S\rangle$, $\supp(S)=\G'$} \}.
    \end{align*}
    
    Note that $\omega(H,S)$ remains the same if we take conjugates of $H$ and $S$, and for any finite $S\sq\mc{G}_{\G}$ we can conjugate $\langle S\rangle$ into $\mc{G}_{\supp(S)}$. The above set is therefore equal to
    \begin{align*}
        \{ \omega(H,S) \mid \text{ $S\sq\mc{G}_{\G}$ finite, $H=\langle S\rangle$, $\supp(S)=\G'$} \},
    \end{align*}
    where we have only replaced ``$S\sq\mc{G}_{\G'}$'' with ``$S\sq\mc{G}_{\G}$''.
    As there are only finitely many subgraphs $\G'\sq\G$, the union of the above sets of growth rates is well-ordered, and the result follows.
\end{proof}

\subsection{Product set growth}\label{sub:product_set_growth}

We can also use our results to obtain lower bounds on the growth of certain subsets of (virtual) graph products. Unless otherwise indicated, we will not assume that the subsets are symmetric; unlike in \Cref{sec:well_ordered}, this does make the results more general and this was an important motivation for this work.

We begin with the subcase of right-angled Artin groups, and more generally virtually special groups. The following generalises Theorem 1.0.4 in \cite{Kerr2021a}, as it covers all subsets rather than just the symmetric case, and gives a dichotomy of subgroups by \cite[Corollary 2.2.2]{Kerr2021a}. 

\begin{thm}
    \label{RaagGrowth}
	Let $G$ be a group that virtually embeds into $A_{\G}$. There exist constants $\alpha,\beta > 0$ such that for every finite $U \subseteq G$, at least one of the following must hold:
	\begin{enumerate}
		\item $\langle U\rangle$ has a finite index subgroup with infinite centre;
		\item $|U^n| \geqslant (\alpha|U|)^{\beta n}$ for every $n \in \mathbb{N}$.
	\end{enumerate}
\end{thm}

\begin{proof}
    This can be proved directly via the same proof format as Theorem 4.3.2 and Corollary 4.3.3 in \cite{Kerr2021a}, using \Cref{ShortStrongIrreducible} from this paper in place of \cite[Corollary 4.2.12]{Kerr2021a}, and \cite[Lemma 5.7]{Wan2023} in place of \cite[Proposition 2.2.19]{Kerr2021a}. Alternatively, it will follow as a corollary of \Cref{GoodImpliesDichotomy} and \Cref{thm:growth_dichotomy}, which generalise the techniques from \cite{Kerr2021a}.
\end{proof}

Similar results have been found in recent years for many classes of groups, including free groups \cite{Safin2011}, hyperbolic groups \cite{Delzant2020}, relatively hyperbolic groups \cite{Cui2021,Wan2023}, Burnside groups of large enough odd exponent \cite{Coulon2022}, mapping class groups \cite{Kerr2021a} (when $U$ is symmetric), and certain manifold groups \cite{Wan2023}. 

The aim of the rest of \Cref{sub:product_set_growth} is to show that graph products of many of these groups satisfy the same dichotomy of subgroups as the one in \Cref{RaagGrowth} (see \Cref{thm:growth_dichotomy} and \Cref{GoodImpliesDichotomy} below).

\subsubsection{\emph{$(\alpha,\beta,N)$--admissible} groups}

We begin by introducing the specific property that we would like our vertex groups to satisfy, in order to obtain the desired dichotomy of subgroups. 

\begin{defn}\label{defn:admissible}
\begin{itemize}
\item[]
    \item A finite subset $U\sq G$ has \emph{$(\alpha,\beta)$--growth} if $|U^n|\geqslant (\alpha|U|)^{\beta n}$ for all $n\in\mathbb{N}$. 
    \item A finitely generated group is \emph{$(\alpha,\beta)$--perfect} if all finite generating sets 
    of its finite-index subgroups have $(\alpha,\beta)$--growth. 
    \item A group $G$ is \emph{$(\alpha,\beta,N)$--admissible} if, for every finitely generated $H\leqslant G$, there exists a subgroup $K\leqslant H$ of index $\leqslant N$ and a torsion-free subgroup of its centre $T\leq Z(K)$ such that $K/T$ is $(\alpha,\beta)$--perfect.
\end{itemize}
\end{defn}

\begin{rem}
    Note that $(\alpha,\beta)$--growth is sometimes refered to as product set growth in the literature (see \cite{Wan2023}, for example), and being $(\alpha,\beta)$--perfect is a stronger property than what is referred to as having (uniform) product set growth.
\end{rem}

Several examples of $(\alpha,\beta,N)$--admissible groups are discussed in \Cref{ex:good} below. We will show that virtual subgroups of graph products of $(\alpha,\beta,N)$--admissible groups satisfy the same dichotomy as found in \Cref{RaagGrowth}. We first collect some useful observations regarding this property.

\begin{rem}
\label{GoodnessRemarks}
The following are easy to see:
\begin{itemize}
    \item If $H$ is a subgroup of $G$, and $G$ is $(\alpha,\beta,N)$--admissible, then $H$ is $(\alpha,\beta,N)$--admissible.
    \item If $H$ is a finite index subgroup of $G$, with index $k$, and $H$ is $(\alpha,\beta,N)$--admissible, then $G$ is $(\alpha,\beta,kN)$--admissible.
    \item If $H$ is a finite index subgroup of $G$, and $G$ is $(\alpha,\beta)$--perfect, then $H$ is $(\alpha,\beta)$--perfect.
\end{itemize}
\end{rem}

\begin{rem}
\label{rem:TreeImpliesGrowth}
    If $G$ is a finitely generated group with a non-elementary acylindrical action on a tree, then $G$ is $(\alpha,\beta)$--perfect for some $\alpha,\beta>0$. This follows from noting that any finite index subgroup of $G$ will also have a non-elementary acylindrical action on the same tree, and using either \cite[Corollary 3.2.20]{Kerr2021a} or \cite[Theorem 1.11]{Delzant2020}, along with \cite[p.\ 64]{Serre1980}.
\end{rem}

\begin{lem}
\label{PerfectFiniteIndex}
    \emph{\cite[Lemma 5.7]{Wan2023}}
    If $H$ is a finite index subgroup of $G$, with index $k$, and $H$ is $(\alpha,\beta)$--perfect, then $G$ is $(\alpha',\beta')$--perfect, with $\alpha',\beta'$ depending only on $\alpha,\beta,$ and $k$.
\end{lem}

\begin{proof}
    Let $K$ be a finite index subgroup of $G$. Then $K\cap H$ has finite index in $H$, so its generating sets have $(\alpha,\beta)$--growth. As $K\cap H$ is also a finite index subgroup of $K$, with index $\leqslant k$, it follows from \cite[Lemma 5.7]{Wan2023} that generating sets of $K$ have $(\alpha',\beta')$--growth, with $\alpha',\beta'$ depending only on $\alpha,\beta,$ and $k$.
\end{proof}

\begin{lem}
\label{PerfectExtends}
    Let $\rho:G\to H$ be an epimorphism with finite kernel.
    \begin{enumerate}
        \item If $H$ is $(\alpha,\beta)$--perfect, then $G$ is $(\alpha/|\ker(\rho)|,\beta)$--perfect.
        \item If $G$ is $(\alpha,\beta)$--perfect, then $H$ is $(\alpha|\ker(\rho)|^{1-\frac{1}{\beta}},\beta)$--perfect.
    \end{enumerate}
\end{lem}

\begin{proof}
    The first part follows from an observation in \cite[Lemma 5.8]{Wan2023}. Note that if $K=\langle U\rangle$ is a finite index subgroup of $G$, then $\rho(K)=\langle \rho(U)\rangle$ is a finite index subgroup of $H$, and the restriction $\rho':K\to \rho(K)$ is an epimorphism with the property that $\ker(\rho')\subseteq\ker(\rho)$. The conclusion then follows from the fact that $|U|\geqslant|\rho(U)|\geqslant \frac{1}{|\ker(\rho')|}|U|$.

    For the second part, note that if $K=\langle V\rangle$ is a finite index subgroup of $H$, then $\rho^{-1}(K)$ is a finite index subgroup of $G$. Let $V=\{v_1,\ldots,v_k\}$, and for each $v_i$ pick some $u_i\in\rho^{-1}(v_i)$. We then observe that $U=\bigcup_{i=1}^k u_i\ker(\rho)$ is a finite generating set of $\rho^{-1}(K)$, with $\rho(U)=V$. We therefore have that for any $n\in\N$
    \begin{align*}
        |V^n|=|\rho(U^n)|\geqslant \frac{|U^n|}{|\ker(\rho)|}\geqslant \frac{(\alpha|U|)^{\beta n}}{|\ker(\rho)|}=\bigg(\frac{\alpha}{|\ker(\rho)|^\frac{1}{\beta n}}|U|\bigg)^{\beta n}\geqslant\bigg(\frac{\alpha}{|\ker(\rho)|^\frac{1}{\beta }}|U|\bigg)^{\beta n},
    \end{align*}
    and the conclusion then follows from the observation that $|U|= |V||\ker(\rho)|$.
\end{proof}

The above observations allow us to see that any $(\alpha,\beta,N)$--admissible group must itself satisfy the dichotomy found in \Cref{RaagGrowth}.

\begin{lem}
\label{GoodImpliesDichotomy}
   Let $G$ be $(\alpha,\beta,N)$--admissible. There exist $\alpha',\beta'>0$ such that for every finitely generated subgroup $H\leqslant G$, exactly one of the following holds:
   \begin{itemize}
       \item $H$ has a finite index subgroup with infinite centre;
       \item $H$ is $(\alpha',\beta')$--perfect.
   \end{itemize}
\end{lem}

\begin{proof}
    Let $K\leqslant H$ be the subgroup of index $\leqslant N$ provided by $(\alpha,\beta,N)$--admissibility, and let $T\leqslant Z(K)$ be the corresponding torsion-free subgroup, so $K/T$ is $(\alpha,\beta)$--perfect. Either $K$ has infinite centre, in which case $H$ virtually does, or we have that $T$ is trivial, so $K$ is $(\alpha,\beta)$--perfect. By \Cref{PerfectFiniteIndex}, $H$ is therefore $(\alpha',\beta')$--perfect, with $\alpha',\beta'$ depending only on $\alpha,\beta,$ and $N$.
\end{proof}

This is a dichotomy of subgroups by the following known result.

\begin{lem}
\label{lem:InfCentre}
    \emph{\cite[Corollary 2.2.2]{Kerr2021a}} If $G$ has a finite index subgroup with infinite centre, then $G$ cannot be $(\alpha,\beta)$--perfect for any $\alpha,\beta>0$.
\end{lem}

For a finitely generated group $G$, denoting by $\mathcal{S}$ its collection of finite generating sets, the \emph{exponential growth rate} is $\omega(G)=\inf_{S\in\mathcal{S}}\omega(G,S)$. A group has \emph{uniform exponential growth} if $\omega(G)>1$. We can see that the finitely generated subgroups of an $(\alpha,\beta,N)$--admissible group additionally satisfy the following dichotomy.

\begin{lem}
\label{GoodImpliesDichotomy2}
   Let $G$ be $(\alpha,\beta,N)$--admissible. There exists $\varepsilon>1$ such that for every finitely generated subgroup $H\leqslant G$, exactly one of the following holds:
   \begin{itemize}
       \item $H$ is virtually abelian;
       \item $\omega(H)>\varepsilon$.
   \end{itemize}
\end{lem}

\begin{proof}
    Let $K\leqslant H$ be the subgroup of index $\leqslant N$ provided by $(\alpha,\beta,N)$--admissibility, and let $T\leqslant Z(K)$ be the corresponding torsion-free subgroup, so $K/T$ is $(\alpha,\beta)$--perfect. If $K/T$ is finite, then $K$ is virtually abelian, so $H$ is virtually abelian. If $K/T$ is infinite, then by \cite[Lemma 2.1.1]{Kerr2021a}, we have that $\omega (K)\geqslant\omega(K/T)\geqslant (1 + \alpha)^{\frac{\beta}{\lceil 1/\alpha\rceil +1}}$. By \cite[Proposition 3.3]{Shalen1992}, $\omega(H)\geqslant (1 + \alpha)^{\frac{\beta}{(\lceil 1/\alpha\rceil +1)(2N+1)}}$.
\end{proof}

We therefore draw our examples of $(\alpha,\beta,N)$--admissible groups from the list of groups that are known to satisfy such dichotomies.

\begin{exm}\label{ex:good}
The following are known examples of groups that are $(\alpha,\beta,N)$--admissible for some $\alpha,\beta,$ and $N$:
\begin{itemize}
    \item \emph{Finite groups.} The trivial group is $(1,1,1)$--admissible, so every finite group $G$ is $(1,1,|G|)$--admissible.
    \item \emph{Virtually abelian groups.} Every free abelian group is $(1,1,1)$--admissible, so every infinite virtually abelian group is $(1,1,N)$--admissible for some $N$.
    \item \emph{Hyperbolic groups.} Admissibility follows from \cite[Theorem 1.1]{Delzant2020} and the fact that there exists an integer $N$ (depending only on the hyperbolicity constant) such that any virtually cyclic subgroup of a hyperbolic group has a torsion-free cyclic subgroup of index $\leq N$.
    \item \emph{Groups hyperbolic relative to $(\alpha,\beta,N)$--admissible groups}, by \cite[Theorem 1.6]{Wan2023}. Note that any finite index subgroup of a non-elementary subgroup is non-elementary; moreover, an elementary subgroup is either conjugate into a peripheral subgroup, in which case we are done by assumption, or it is virtually cyclic and we can conclude by a similar argument to the hyperbolic case (using \cite[Lemma 6.8]{Osin2016} and \cite[Theorem 4.2]{Osin2004}).
    \item \emph{Burnside groups of sufficiently large odd exponent}, by \cite[Theorem 1.2]{Coulon2022}, noting that the cited statement implies that any abelian subgroup must be finite, and that any finite subgroup must have bounded size.
    \item \emph{Right-angled Artin groups}, by \Cref{RaagGrowth} and following the proof of Theorem 4.3.2 in \cite{Kerr2021a}. In particular, a finitely generated subgroup is either $(\alpha,\beta)$--perfect, or it has infinite centre, with the centre coming from being a subgroup of some $H\times\mathbb{Z}^n$. In particular, the centre is the intersection with $\mathbb{Z}^n$, and if we quotient by this we are back in the $(\alpha,\beta)$--perfect case.
    \item \emph{Virtually special groups}, by \Cref{GoodnessRemarks} and the fact that every virtually special group virtually embeds in a right-angled Artin group \cite{Haglund-Wise2008}.
    \item \emph{Dyer groups}, by the fact that they are subgroups of Coxeter groups \cite[Corollary 1.2]{Soergel2024}, which are themselves virtually (non-compact) special \cite{Haglund2010}.
    \item \emph{Mapping class groups} (restricting to symmetric subsets $U$), following the proof of Theorem 4.4.1 in \cite{Kerr2021a}. By \Cref{GoodnessRemarks}, it suffices to check this for a pure finite index subgroup, and similarly to the right-angled Artin group case, every subgroup is either $(\alpha,\beta)$--perfect or can be quotiented by its centre to obtain an $(\alpha,\beta)$--perfect subgroup of another mapping class group.
    \item \emph{Free-by-cyclic groups.} Let $G$ be a free-by-cyclic group (taking a finite index subgroup if necessary), and let $H$ be a finitely generated subgroup of $G$. The cases to consider essentially follow the cases given in the proof of Corollary 1.2 in \cite{Kudlinska2024}, and the references therein. 
    
    Case 1: $G$ has exponentially growing monodromy, and is hyperbolic relative to free-by-cyclic groups with polynomially growing monodromy. If $H$ is not peripheral, we are done by the relatively hyperbolic case. If $H$ is peripheral, then we are reduced to the case where $G$ has polynomially growing monodromy.

    Case 2: $G$ has polynomially growing monodromy, and has a non-elementary acylindrical action on a simplicial tree, without edge inversions. If $H$ has a non-elementary action on the tree, then we can apply \Cref{rem:TreeImpliesGrowth}. If $H$ has an elementary action with infinite orbits, it is virtually cyclic, with a torsion-free cyclic subgroup of bounded index. If $H$ stabilises a vertex, then it is a subgroup of a free-by-cyclic groups with polynomially growing monodromy of strictly lower degree than $G$, so we can use induction.

    Case 3: G has polynomially growing monodromy, and has a torsion-free central subgroup $T$ such that $G/T$ is virtually free. Note $H\cap T$ is torsion-free and central in $H$, with $H/(H\cap T)$ virtually free, so we are in the hyperbolic case.
\end{itemize}
\end{exm}

\begin{exm}
    We note here that there are groups that satisfy the dichotomy stated in \Cref{RaagGrowth}, but fail to be admissible. The fundamental group $G$ of any complete Riemannian manifold with pinched negative curvature satisfies the dichotomy by \cite[Theorem 1.5]{Wan2023}, noting that virtually nilpotent groups all have a finite index torsion-free nilpotent subgroup \cite[pg.~2]{segal2005polycyclic}, which will either be trivial or have infinite centre. On the other hand, taking $G$ to be a non-uniform lattice in $ {\rm PU}(2,1)$, its parabolic subgroups $H$ are nilpotent and not virtually abelian. Thus, even passing to finite index and quotienting out the centre, $H$ can never be $(\alpha,\beta)$--perfect. It also clearly cannot satisfy the dichotomy in \Cref{GoodImpliesDichotomy2}. Consequently, $G$ is not $(\alpha,\beta,N)$--admissible for any constants $\alpha, \beta$, and $N$. 
\end{exm}

\subsubsection{Growth dichotomies for graph products}

We need a couple of preliminary results before we prove the main theorem of this section. In particular, these results are for subgroups of graph products whose essential supports are neither a single vertex nor a join. Note that in these cases, we do not need to make any assumptions about the vertex groups.

The following is \Cref{corintro:growth1} in the introduction. 

\begin{prop}
\label{GPGrowth3}
	There exist constants $\alpha,\beta > 0$, only dependent on $\G$, such that for every finitely generated $H\leqslant \mc{G}_{\G}$, where $\emph{esupp}(H)$ is neither a single vertex nor a join, exactly one of the following holds:
    \begin{itemize}
        \item $H$ is isomorphic to $\Z$ or $D_{\infty}$;
        \item $H$ is $(\alpha,\beta)$--perfect.
    \end{itemize}
    In both cases, there exists a subgroup $K\leqslant H$ of index $\leqslant 2$ such that $K/Z(K)$ is $(\alpha,\beta)$--perfect. In addition, $Z(K)$ is torsion-free.
\end{prop}

\begin{proof}
    Suppose first that $H$ is virtually $\Z$. By Theorem 4.1 in \cite{Antolin2015}, $H$ is isomorphic to either $\Z$ or $D_{\infty}$. Hence $H$ has a subgroup of index $\leqslant 2$ isomorphic to $\Z$, so the conclusion holds.

    Now suppose $H$ is not virtually $\Z$. There exists $g\in H$ such that $\supp(g)=\supp(H)$, and such that the centraliser of $g$ in $H$ is isomorphic to $\Z$ \cite[Lemma 6.17]{Minasyan2015}\cite[Theorem 56]{Barkauskas2007}. As $H$ is acylindrically hyperbolic, it has finite centre \cite[Corollary 7.2]{Osin2016}. Given that the centraliser of $g$ is torsion free, $H$ must have trivial centre. We therefore want to show that $H$ is $(\alpha,\beta)$--perfect.

    Let $K$ be a finite index subgroup of $H$. Note that some power of $g$ lies in $K$, and as $\supp(g)$ has no cone vertices, we also know that $\supp(g)=\stsupp(g)$. Therefore $\supp(K)=\supp(H)$, and in particular is neither a single vertex nor a join.

    Up to conjugation, we can assume that $K\leqslant\mc{G}_{\supp(K)}$. Let $U$ be a finite generating set of $K$. The case that $\supp(K)=\emptyset$ is trivial, so we assume that $\supp(K)\neq\emptyset$. Repeating the argument from \Cref{WellOrdered}, we apply \Cref{ShortStrongIrreducible} to find $N=N(|\G^{(0)}|)\in\mathbb{N}$ such that there exists $n\leqslant N$ with $U^n$ containing a strongly irreducible element in $\mc{G}_{\supp(K)}$. By \Cref{StronglyIrreducibleIsLoxodromic}, this element is loxodromic in the action of $\mc{G}_{\supp(K)}$ on the contact graph $C(\mc{G}_{\supp(K)})$. This action is acylindrical, and the contact graph is a quasi-tree \cite[Corollary C]{Valiunas2021}. In particular, the acylindricity constants and the quasi-isometry constants only depend on $\supp(K)$ \cite[Theorem A and Theorem 4.6]{Valiunas2021}. As $K$ is not virtually $\mathbb{Z}$, it follows from \cite[Corollary 3.2.20]{Kerr2021a} that there exist $\alpha,\beta > 0$, depending only on $\supp(K)$ and $N$, such that $|U^n| \geqslant (\alpha|U|)^{\beta n}$ for every $n \in \mathbb{N}$.

    As there are only finitely many non-join subgraphs of $\G$, we can take the infimum to get some $\alpha,\beta>0$ that work for all such $U\subseteq\mc{G}_{\G}$.
\end{proof}

We are now able to prove \Cref{thmintro:growth_dichotomy}, which we restate here.

\begin{thm}\label{thm:growth_dichotomy}
    Let $\mc{G}_{\G}$ be a graph product of $(\alpha,\beta,N)$--admissible groups, and let $G$ be a group that virtually embeds into $\mc{G}_{\G}$. There exist $\alpha',\beta',N'>0$ such that $G$ is $(\alpha',\beta',N')$--admissible.
\end{thm}

\begin{proof}
    We assume that $N\geqslant 2$ for ease of notation (otherwise, replace $N$ by 2 in what follows). We note that, by \Cref{GoodnessRemarks}, it is sufficient to show that $\mc{G}_{\G}$ is $(\alpha',\beta',N')$--admissible for some $\alpha',\beta',N'>0$.

    We therefore let $H\leqslant \mc{G}_{\G}$ be finitely generated, and suppose, up to replacing $H$ with a conjugate, that $H\leqslant \mc{G}_{\supp(H)}$. If $\supp(H)$ is a single vertex, then we conclude by using that $H$ is a subgroup of one of the $(\alpha,\beta,N)$--admissible vertex groups. If $\supp(H)$ is neither a single vertex nor a join, then we conclude using \Cref{GPGrowth3}, noting in particular that the $\alpha',\beta'$ from this result are dependent only on $\G$. Let $\alpha''=\min\{\alpha,\alpha',1\}$, and let $\beta''=\min\{\beta,\beta'\}$. Note that $\alpha''\leqslant 1$ covers the case where $\supp(H)=\emptyset$, in other words the case where $H$ is trivial.

    We are only left to consider the situation in which $\supp(H)$ is a join. We can therefore write
    \begin{equation*}
        \supp(H)=\{v_1\}\ast\cdots\ast\{v_k\}\ast\G_{k+1}\ast\cdots\ast\G_m
    \end{equation*}
    for some $k\geqslant 0$ and $m\geqslant 2$, where each $v_i$ is a vertex of $\G$, and each $\G_i$ is neither a single vertex nor a join. Therefore
    \begin{equation*}
        H\leqslant \mc{G}_{v_1}\times\cdots\times \mc{G}_{v_k}\times\mc{G}_{\G_{k+1}}\times\cdots\times\mc{G}_{\G_m}.
    \end{equation*}
    Note that $m$ is no bigger than the maximal size of a clique in $\G$.

    Let $H_i$ be the projection of $H$ to the $i$th factor, and note that in each case $H_i$ has full support in that factor. For $i\leqslant k$, let $K_i\leqslant H_i$ be the subgroup of index $\leqslant N$ provided by $(\alpha,\beta,N)$--admissibility, and let $T_i\leqslant Z(K_i)$ be the corresponding torsion-free subgroup of its centre. For $i\geqslant k+1$, let $K_i\leqslant H_i$ be the subgroup of index $\leqslant 2$ provided by \Cref{GPGrowth3}, and let $T_i=Z(K_i)$. Then, setting $K:=K_1\x\cdots\x K_m$, the intersection $K\cap H$ has index $\leqslant N^m$ in $H$.

    Let $T:=(T_1\times\cdots\times T_m)\cap H$, which is a torsion-free central subgroup of $K\cap H$. We want to show that $(K\cap H)/T$ is $(\alpha''',\beta''')$-perfect for some $\alpha''',\beta'''>0$. Note that every finite index subgroup of $(K\cap H)/T$ is of the form $L/(L\cap T)$ for some finite index subgroup $L$ of $K\cap H$.
    
    Let $L_i$ be the projection of $L$ to $K_i$. We can see that the projection of $K\cap H$ to each $K_i$ is a finite index subgroup of that $K_i$, so $L_i$ is also a finite index subgroup of $K_i$. Let $P_i\colon K_i\ra K_i/T_i$ be the quotient maps obtained by taking the quotient of each $K_i$ by the corresponding $T_i$. Then $P_i(L_i)$ is a finite index subgroup of $K_i/T_i$. As $K_i/T_i$ is $(\alpha'',\beta'')$--perfect, by construction, any finite generating set of $P_i(L_i)$ has $(\alpha'',\beta'')$--growth.

    Consider the natural map $L\ra K_1/T_1\times\cdots\times K_m/T_m$, which has kernel $L\cap T$, and therefore induces an injective map $\pi:L/(L\cap T)\hookrightarrow K_1/T_1\times\cdots\times K_m/T_m$. Note that the projection of $\pi(L/(L\cap T))$ to $K_i/T_i$ is exactly $P_i(L_i)$. Therefore for any finite generating set $W$ of $L/(L\cap T)$, the projection of $\pi(W)$ to each $K_i/T_i$ has $(\alpha'',\beta'')$--growth, so by \cite[Corollary 2.2.10]{Kerr2021a}, $\pi(W)$ has $(\alpha'',\frac{\beta''}{m})$--growth. As $\pi$ is injective, $W$ has $(\alpha'',\frac{\beta''}{m})$--growth. We conclude that $(K\cap H)/T$ is $(\alpha'',\frac{\beta''}{m})$-perfect, so $\mc{G}_{\G}$ is $(\alpha'',\frac{\beta''}{m},N^m)$--admissible.
\end{proof}

\subsection{Effective Tits alternatives} \label{sec:TitsAlternative}

It was proved in \cite{Antolin2015} that if the vertex groups of a graph product satisfy one of several versions of the Tits alternative, then the same is true of the graph product itself. We show here that a corollary of \Cref{thm:short_loxodromics} is that the same is true for effective versions of the Tits alternative. The following definitions are adapted from \cite{Antolin2015}.

\begin{defn}
    Let $\mc{I}$ be a collection of cardinals. A group $G$ is $\mc{I}$\textit{-generated} if there is a generating set $S$ of $G$ and $\lambda \in\mc{I}$ such that $|S| \leqslant \lambda$. We call $S$ an $\mc{I}$\textit{-generating set} of $G$.
\end{defn}
    
\begin{defn}
    Suppose that $\mc{I}$ is a collection of cardinals, $\mc{C}$ is a class of groups, and $G$ is a group. We say that $G$ satisfies the \textit{effective Tits alternative relative to} $(\mc{I}, \mc{C})$ if there exists $N\in\mathbb{N}$ such that for any $\mc{I}$-generated subgroup $H\leqslant G$, either $H\in\mc{C}$, or for any $\mc{I}$-generating set $S$ of $H$, there exist $a,b\in B_S(N)$ such that $\langle a,b\rangle\cong F_2$. In this second case, we say that a free subgroup can be $\mc{I}$\textit{-generated effectively} in $H$.
\end{defn}

For this version of the Tits alternative to hold for graph products, the
collection of cardinals $\mc{I}$ and class of groups $\mc{C}$ must satisfy the following properties, taken from \cite{Antolin2015}:

\begin{itemize}
    \item[(P0)] If $A$ and $B$ are groups, with $A\in\mc{C}$ and $A\cong B$, then $B\in\mc{C}$.
    \item[(P1)] If $A\in\mc{C}$ and $B\leqslant A$ is an $\mc{I}$-generated subgroup, then $B\in\mc{C}$.
    \item[(P2)] If $A, B \in\mc{C}$ are $\mc{I}$-generated, then $A \times B \in\mc{C}$.
    \item[(P3)] $\mathbb{Z}\in\mc{C}$.
    \item[(P4)] If $\mathbb{Z}/2\mathbb{Z}\in\mc{C}$, then $D_{\infty}\in\mc{C}$.
\end{itemize}

As noted in \cite{Antolin2015}, when $\mc{I}$ is the collection of finite cardinals, the above properties are satisfied for several natural choices of $\mc{C}$, for example when it is the class of virtually abelian groups, virtually nilpotent groups, or elementary amenable groups.

\begin{thm}
    Suppose that $\mc{I}$ is a collection of cardinals, and $\mc{C}$ is a class of groups satisfying conditions (P0)–(P4). Let $\G$ be a finite graph, and let $\mc{G}_{\G}$ be a graph product of groups that satisfy the effective Tits Alternative relative to $(\mc{I}, \mc{C})$. Then $\mc{G}_{\G}$ satisfies the effective Tits Alternative relative to $(\mc{I}, \mc{C})$.
\end{thm}

\begin{proof}
    The proof is almost identical to the proof of Theorem A in \cite{Antolin2015}, which we sketch here. Their proof goes via induction on $|\G^{(0)}|$, with two cases to consider. The first is that $\G$ is a join, and so $\mc{G}_{\G}=\mc{G}_{\G_1}\times\mc{G}_{\G_2}$ with $|\G_1^{(0)}|,|\G_2^{(0)}|<|\G^{(0)}|$. In this case, the projections of any $\mc{I}$-generated subgroup $H$ of $\mc{G}_{\G}$ to the factors $\mc{G}_{\G_1}$ and $\mc{G}_{\G_2}$ are either both in $\mc{C}$, in which case we can apply (P1) and (P2) to show that $H\in\mc{C}$, or the inductive hypothesis tells us that a free subgroup can be $\mc{I}$-generated effectively in one of the projections of $H$, which means that a free subgroup can be $\mc{I}$-generated effectively in $H$ itself.
    
    The second case is that $\G$ is not a join. If $|\G^{(0)}|=1$, then we are in the base case, so assume that $|\G^{(0)}|\geqslant 2$. If any $\mc{I}$-generated subgroup $H$ of $\mc{G}_{\G}$ is isomorphic to a subgroup of some $\mc{G}_{\G'}$ with $|\G'^{(0)}|<|\G^{(0)}|$, then we are done by induction. Otherwise, \cite[Corollary 4.1]{Antolin2015}, \Cref{GPGrowth3}, and \cite[Corollary 5.9]{Antolin2015} say that either $H\cong \mathbb{Z}$ (so we apply (P3)), $H\cong D_{\infty}$ (so we apply (P4)), or  there exists $N\in\mathbb{N}$ such that for any generating set $S$ of the subgroup $H$, there exist $a,b\in B_S(N)$ such that $\langle a,b\rangle\cong F_2$. The effective part of this follows from the reasoning in \Cref{GPGrowth3}, where we showed that there exists a uniform $N'\in\mathbb{N}$ such that there exists a loxodromic $a'\in B_S(N')$. We can then apply the well known fact that in an acylindrically hyperbolic group this can be used to find the $a,b,$ and $N$ above (see for example \cite[Corollary 4.2.2]{Ng2020}).
\end{proof}

\appendix

\section{Appendix}

This appendix collects two results that, although not required in the main body of the paper, are in a rather similar spirit. 

\Cref{Appendix1} gives an alternative argument for constructing short simultaneous loxodromics on finitely many trees. We allow more general trees than the Bass--Serre trees considered in \Cref{sect:graph_product_main}, but obtain worse bounds on the length of the required words (for graph products, these would require replacing $\dim(\G)$ with $|\G^{(0)}|$). 

\Cref{Appendix2} records an argument due to Carolyn Abbott and Thomas Ng for constructing short elements of full support in right-angled Artin groups $A_{\G}$. This has the advantage of being much simpler than our proofs, but it only applies to generating sets of the entire $A_{\G}$, rather than those of arbitrary subgroups.

\subsection{Simultaneous loxodromics on trees}
\label{Appendix1}

An important step in this paper is \Cref{lem:combining_supp}, where we show that we can (partially) combine the stable supports of two elements in a graph product uniformly quickly. The main strategy in proving this is to combine the two elements in such a way that the resulting group element is loxodromic on all of the relevant Bass--Serre trees simultaneously. To do this, the bounded creasing property is used, which is specific to standard Bass--Serre trees of graph products. 

In the absence of the bounded creasing property, it is still possible to uniformly quickly obtain elements that are simultaneously loxodromic on a collection of trees, although the constants obtained will not be as good in general. We state and prove this result here, as it may be of independent interest. 

\begin{thm}
	\label{SimulSupport}
	For any $k\in\mathbb{N}$, there exists a constant $M=M(k)\in\mathbb{N}$ such that, given the following:
	\begin{itemize}
		\item a collection of trees $T_1,\ldots,T_k$,
		\item a group $G$ that acts on each $T_i$ by isometries,
		\item elements $g,h\in G$ such that for each $T_i$ at least one of $g,h$ is loxodromic on $T_i$,
	\end{itemize}
	then there exist $1\leqslant m,n\leqslant M$ such that $g^mh^n$ is loxodromic on every $T_i$. In particular, we can take $M=(2k)^k(2k+1)$. 
\end{thm}

\begin{proof}
    Given $g\in G$, we will use $\tau_i(g)$ to refer to the translation length of $g$ in $T_i$, $A_i(g)$ for the axis of $g$ in $T_i$ if $\tau_i(g)>0$, and $\Fix_i(g)$ for the fixed set of $g$ in $T_i$ if $\tau_i(g)=0$.
 
    Let $g,h\in G$ such that for each $T_i$ at least one of $g,h$ is loxodromic on $T_i$. We have a partition $\{1,\ldots,k\}=\Sigma_{e\ell}\sqcup\Sigma_{\ell e}\sqcup\Sigma_{\ell\ell}$, where $i\in\Sigma_{e\ell}$ if $g$ is elliptic on $T_i$ and $h$ is loxodromic, $i\in\Sigma_{\ell e}$ if $g$ is loxodromic on $T_i$ and $h$ is elliptic, and $i\in\Sigma_{\ell\ell}$ if both $g$ and $h$ are loxodromic on $T_i$. 
    
    We will raise $g$ and $h$ to sufficiently high powers in order to avoid the cases in \Cref{prop:elliptic_product}. For $i\in\Sigma_{e\ell}\cup\Sigma_{\ell e}$, we need to avoid the situation detailed in \Cref{rmk:creasing_types}(a): when the elliptic element fixes a single point of the loxodromic axis and creases an arc of it, in the negative direction. For $i\in\Sigma_{\ell\ell}$, we need to avoid both the situation in \Cref{rmk:creasing_types}(b) (when the shorter loxodromic creases the axis of the longer one) and the situation in \Cref{prop:elliptic_product}(3) (when the two loxodromics have equal translation lengths). 

    \smallskip
    {\bf Step~1.} \emph{There exists $1\leq m'\leq (2k)^k$ such that there are at least $k+1$ integers $1\leq a\leq 2k+1$ such that $g^{m'a}$ does not crease any arc of $A_i(h)$ in the negative direction for any $i\in\Sigma_{e\ell}\cup\Sigma_{\ell\ell}$.}
	 
	After re-ordering, we can assume that $\Sigma_{e\ell}=\{1,\ldots, p\}$ (if non-empty). Suppose that for some $i\in\{1,\ldots,p\}$ there exists $m_i\in\{1,\ldots,2k\}$ such that $\text{diam}(\Fix_{i}(g^{m_i})\cap A_{i}(h))>0$. After re-ordering, we can assume that $i=1$. Note that any powers of $g^{m_1}$ and $h$ will retain this property, and so no such power of $g$ will crease an arc of $A_1(h)$. We therefore repeat this process, beginning with the action of $g^{m_1}$ on $T_2$, up to some maximum possible $p'\in\{1,\ldots,p\}$. We thus obtain $m_1,\ldots, m_{p'}\leqslant 2k$ such that $\text{diam}(\Fix_{i}(g^{m_1\cdots m_{p'}})\cap A_{i}(h))>0$ for every $i\in\{1,\ldots,p'\}$. Let $m'=m_1\cdots m_{p'}$.
	
	We suppose that we cannot continue any further, so for all $i\in\{p'+1,\ldots,p\}$ (if non-empty) and for all $m''\in\{1,\ldots,2k\}$ we have that $\Fix_{i}(g^{m'm''})$ and $ A_{i}(h)$ are disjoint, or $\Fix_{i}(g^{m'm''})\cap A_{i}(h)$ is a single vertex.

    Now, for each $i\in\Sigma_{e\ell}$ there exists at most one integer $a\in\{1,\ldots,2k+1\}$ such that $g^{m'a}$ creases an arc of $A_i(h)$ in the negative direction. Indeed, note that creasing can only occur for $i\in\{p'+1,\dots,p\}$. If there existed distinct $a,b\in\{1,\ldots,2k+1\}$ such that $g^{m'a}$ and $g^{m'b}$ crease the same arc $\alpha$ of $A_i(h)$ then, assuming without loss of generality that $a<b$, the element $g^{m'(b-a)}$ would stabilise $\alpha$ contradicting our assumption about $i$. 
    
    Similarly, for each $i\in\Sigma_{\ell\ell}$ there exists at most one integer $a$ such that $g^{m'a}$ creases an arc of $A_i(h)$, as this can only occur when $\tau_i(g^{m'a})=a\cdot\tau_i(g^{m'})$ equals $\ell(A_i(g)\cap A_i(h))$. 

    We therefore have that there are at most $k$ integers $a\in\{1,\ldots,2k+1\}$ such that $g^{m'a}$ creases an arc of $A_i(h)$ in the negative direction for some $i\in\Sigma_{e\ell}\cup\Sigma_{\ell\ell}$. This implies that there are at least $k+1$ integers $a\in\{1,\ldots,2k+1\}$ such that no creasing occurs.

    \smallskip
    {\bf Step~2.} \emph{There exists $1\leq n'\leq k^k$ such that there exists at least one integer $1\leq b\leq k+1$ such that $h^{n'b}$ does not crease any arc of $A_i(g)$ for any $i\in\Sigma_{\ell e}\cup\Sigma_{\ell\ell}$.}
    
	We proceed as in Step~1. Re-order the trees such that $\Sigma_{\ell e}=\{1,\ldots,q\}$ (if non-empty), and using the same procedure find a maximal $q'\in\{1,\ldots, q\}$ and $n'=n_1\cdots n_{q'}$ such that $n_i\in\{1,\ldots,k\}$ and $\text{diam}(A_{i}(g)\cap\Fix_{i}(h^{n'}))>0$ for every $i\in\{1,\ldots,q'\}$. It is then the case that there is at least one $b\in\{1,\ldots, k+1\}$ such that $h^{n'b}$ does not crease an arc of $A_i(g)$ in the negative direction for any $i\in\{q'+1,\ldots,q\}\cup\Sigma_{\ell\ell}$, which guarantees no creasing for $i\in\Sigma_{\ell e}\cup\Sigma_{\ell\ell}$.

    \smallskip
    {\bf Step~3.} \emph{There exist integers $1\leq m\leq (2k)^k(2k+1)$ and $1\leq n\leq k^k(k+1)$ such that $g^mh^n$ is loxodromic on all $T_i$.}

    Let $m',n'$ and $a_1,\dots,a_{k+1},b$ be the integers provided by the previous two steps. For each $j\in\{1,\dots,k+1\}$, the product $g^{m'a_j}h^{n'b}$ is loxodromic on all trees $T_i$ with $i\in\Sigma_{e\ell}\cup\Sigma_{\ell e}$. We are left to choose $j$ so that this product is also loxodromic on all trees $T_i$ with $i\in\Sigma_{\ell\ell}$. For this, it suffices to ensure that $\tau_i(g^{m'a_j})=a_j\cdot \tau(g^{m'})$ is different from $\tau_i(h^{n'b})$ for all $i\in\Sigma_{\ell\ell}$, as the previous two steps have already ruled out any form of creasing. Since $|\Sigma_{\ell\ell}|\leq k$, one of the $k+1$ integers $a_j$ will do the job.

    This completes Step~3 and concludes the proof of the proposition.
\end{proof}

Using the same method as in the proof of \Cref{prop:full_support}, with Serre's lemma \cite[p.\ 64]{Serre1980} in place of \Cref{lem:U^2_covers_support}, we can extend \Cref{SimulSupport} to get an analogue of \Cref{cor:torsion_free_full_support} for loxodromic elements on trees.

\begin{cor}
	\label{cor:SimulSupport}
    Let $G$ be a group, let $T_1,\ldots,T_k$ be a collection of trees, and suppose that $G$ acts on each $T_i$ by isometries. For a subset $U\sq G$, let 
    \[\emph{supp}(U):=\{T_i\colon\langle U\rangle\emph{ contains a loxodromic element on }T_i\}.\] 
    Then there exists a constant $N=N(k)\in\mathbb{N}$ such that for every $U\sq G$, there exists an integer $1\leqslant n\leqslant N$ such that there exists $g\in U^n$ that is loxodromic on every $T_i\in\emph{supp}(U)$.
\end{cor}

A similar result to \Cref{SimulSupport} is found for actions on a collection of hyperbolic spaces in \cite{Clay2018}, with the difference being that we are able to obtain this simultaneous loxodromic uniformly quickly. This naturally leads us to ask the following:

\begin{ques}
    Does \Cref{SimulSupport} hold for isometries of hyperbolic spaces? 
    
    More precisely, does there exist a constant $M=M(k,\delta,\varepsilon)$ such that, if $X_1,\dots,X_k$ are geodesic $\delta$--hyperbolic spaces on which two elements $g,h$ act, if at least one of $g,h$ is loxodromic on each $X_i$, and if the stable translation lengths of $g,h$ on each $X_i$ are either $0$ or $\geq\varepsilon$, then there exist $1\leq m,n\leq M$ such that $g^mh^n$ is loxodromic on every $X_i$? 
\end{ques}

In relation to this question, we remark that, for every integer $N\in\N$, there exists a pair $g_N,h_N$ of isometries of $\mathbb{H}^4$ such that the group $\langle g_N,h_N\rangle$ contains loxodromic isometries of $\mathbb{H}^4$, but every element of the ball $\{{\rm id},g_N^{\pm},h_N^{\pm}\}^N$ is elliptic in $\mathbb{H}^4$ \cite[Lemma~1.12]{Breuillard2021}. This should act as a cautionary example showing that transferring similar results from trees to hyperbolic spaces is not always possible (in this case, Serre's lemma \cite[p.\ 64]{Serre1980}). 

\subsection{A previously existing method for right-angled Artin groups}
\label{Appendix2}

There also exists a direct combinatorial method for finding elements of full support uniformly quickly in right-angled Artin groups, although it does not apply to all subgroups. The following was communicated to us by Carolyn Abbott and Thomas Ng, and predates the other results in this paper. With their permission, we include here their short proof of this result.

\begin{thm}
    \label{ThomasCarolyn}
    Let $\Gamma$ be a finite graph. There exists a constant $N=N(|\G^{(0)}|)\in\mathbb{N}$ such that for every generating set $S$ of $A_{\G}$, there exists $n\leqslant N$ and $g\in S^n$ such that $\emph{esupp}(g)=\G$.
\end{thm}

This can be proved as an easy consequence of \Cref{ThomasCarolyn2}, below.

\begin{note}
    Let $\G^{(0)}=\{v_1,\ldots,v_k\}$. For $i\in\{1,\ldots,k\}$, let $p_i:A_{\G}\to \mathbb{Z}$ be the sum of the powers of $v_i$ that appear in $g\in A_{\G}$, when $g$ is written as a reduced word.
\end{note}

Note that for $g,g'\in A_{\G}$, $n\in\mathbb{N}$, we have that $p_i(gg')=p_i(g)+p_i(g')$ and $p_i(g^n)=np_i(g)$. We can also note that $\{v_i\in \G^{(0)}:p_i(g)\neq 0\}\sq\text{esupp}(g)$.

\begin{prop}
\label{ThomasCarolyn2}
    Let $\G$ be a finite graph. For $U\sq A_{\G}$, let 
    \[\G_U=\{v_i\in \G^{(0)}:\exists u\in U, p_i(u)\neq 0\}.\] 
    There exists a constant $N=N(|\G_U|)\in\mathbb{N}$ such that there is some $n\leqslant N$ and $g\in U^n$ with $\G_U\sq\emph{esupp}(g)$.
\end{prop}

\begin{proof}
    Assume $\G_U$ is non-empty. Re-order $\G^{(0)}$ such that $\G_U=\{v_1,\ldots,v_m\}$. By our definition of $\G_U$, for every $i\in\{1,\ldots,m\}$ there exists $u_i\in U$ such that $p_i(u_i)\neq 0$. We now want to find bounded $n_1,\ldots,n_m\in\mathbb{N}$ such that $p_i(u_1^{n_1}\cdots u_m^{n_m})\neq 0$ for every $i\in\{1,\ldots,m\}$. We will do this by induction. When $m=1$, we can take $n_1=1$, and $p_1(u_1)\neq 0$ by definition.

    Now suppose that we have found $n_1,\ldots,n_j\in\mathbb{N}$ such that $p_i(u_1^{n_1}\cdots u_j^{n_j})\neq 0$ for every $i\in\{1,\ldots,j\}$. For a given $i\in\{1,\ldots,j+1\}$, and $n_{j+1}\in\mathbb{N}$, we can see that $p_i(u_1^{n_1}\cdots u_{j+1}^{n_{j+1}})= 0$ if and only if $n_{j+1}p_i(u_{j+1})=-p_i(u_1^{n_1}\cdots u_j^{n_j})$.

    As $p_i(u_1^{n_1}\cdots u_j^{n_j})\neq 0$ for every $i\in\{1,\ldots,j\}$, and $p_{j+1}(u_{j+1})\neq 0$, we have that at most one value of $n_{j+1}\in\mathbb{N}$ can satisfy $n_{j+1}p_i(u_{j+1})=-p_i(u_1^{n_1}\cdots u_j^{n_j})$ for each $i\in\{1,\ldots,j+1\}$. It therefore follows that there exists $n_{j+1}\in \{1,\ldots,j+1\}$ such that $p_i(u_1^{n_1}\cdots u_{j+1}^{n_{j+1}})\neq 0$ for each $i\in\{1,\ldots,j+1\}$.

    By induction, we can find $n_1,\ldots,n_m\in\mathbb{N}$ such that $n_j\leqslant j+1$ for every $j\in\{1,\ldots,m\}$, and which satisfy that $p_i(u_1^{n_1}\cdots u_m^{n_m})\neq 0$ for every $i\in\{1,\ldots,m\}$. Let $N=\frac{(m+1)(m+2)}{2}$, then we have found $n\leqslant N$ and $g\in U^n$ such that such that $\G_U\sq\text{esupp}(g)$.
\end{proof}

\begin{proof}[Proof of \Cref{ThomasCarolyn}]
    As $S$ is a generating set of $A_{\G}$, in particular it must generate every element of $V=\{v_1,\ldots,v_k\}$. For any $i\in\{1,\ldots,k\}$, if $p_i(s)= 0$ for every $s\in S$, then the same would be true of their products, which would be a contradiction. In the notation of \Cref{ThomasCarolyn2}, we have that $\G_S=\G^{(0)}$.
\end{proof}

 \bibliography{references.bib}
 \bibliographystyle{alpha}
	
\end{document}